\documentclass{amsart}
\usepackage[margin = 1.3 in]{geometry}
\usepackage{amsmath,amssymb,amsthm,tikz,hyperref,mathrsfs,color,thmtools,caption,enumitem,mathtools,comment}
\newcommand{\nc}{\newcommand}
\nc{\dmo}{\DeclareMathOperator}
\dmo{\ra}{\rightarrow}
\dmo{\Prob}{\mathbb{P}}
\dmo{\E}{\mathbb{E}}
\dmo{\N}{\mathbb{N}}
\dmo{\Z}{\mathbb{Z}}
\dmo{\Q}{\mathbb{Q}}
\dmo{\R}{\mathbb{R}}
\dmo{\C}{\mathcal{C}}
\dmo{\w}{\omega}
\dmo{\diam}{\operatorname{diam}}
\dmo{\Out}{\operatorname{Out}}
\dmo{\Mod}{\operatorname{Mod}}
\dmo{\proj}{\operatorname{proj}}
\dmo{\supp}{\operatorname{supp}}
\nc{\nt}{\newtheorem}

\nt{theorem}{Theorem}

\newtheorem{thm}{{\bf Theorem}}[section]

\newtheorem{lem}[thm]{{\bf Lemma}}
\newtheorem{cor}[thm]{{\bf Corollary}}
\newtheorem{prop}[thm]{{\bf Proposition}}

\newtheorem{claim}[thm]{Claim}

\newtheorem{definition}[thm]{Definition}

\numberwithin{equation}{section}

\begin{document}

\title{Random walks on groups and superlinear-divergent geodesics}

\author   {Kunal Chawla}
 \address{Department of Mathematics, Princeton University, Princeton, NJ}
 \email{kc7106@princeton.edu}

\author   {Inhyeok Choi}
 \address{June E Huh Center for Mathematical Challenges of KIAS, Seoul, South Korea}
 \email{inhyeokchoi48@gmail.com} 

 \author   {Vivian He}
 \address{Department of Mathematics, University of Toronto, Toronto, ON }
 \email{vivian.he@mail.utoronto.ca} 
 
 \author   {Kasra Rafi}
 \address{Department of Mathematics, University of Toronto, Toronto, ON }
 \email{rafi@math.toronto.edu}

\maketitle

\begin{abstract}
In this paper, we study random walks on groups that contain superlinear-divergent geodesics, in the line of thoughts of Goldsborough-Sisto. The existence of a superlinear-divergent geodesic is a quasi-isometry invariant which allows us to execute Gou{\"e}zel's pivoting technique. We develop the theory of superlinear divergence and establish a central limit theorem for random walks on these groups. 
\end{abstract}

\section{Introduction}
Classical limit laws in probability theory concern the asymptotic behaviour of the random variable 
\[
Z_{n} = \xi_{1} + \xi_{2} + \cdots+ \xi_{n}.
\]
for i.i.d. random variables $\xi_{1}, \xi_{2}, \ldots$ on $\mathbb{R}$.
As a non-commuting counterpart, Bellman, Furstenberg and Kesten initated the study of random walks on a matrix group $G$  (\cite{bellman1954limit}, \cite{kesten1959symmetric}, \cite{furstenberg1960products}, \cite{furstenberg1963noncommuting}). Given a probability measure $\mu$ on $G$, the random walk generated by $\mu$ is a Markov chain on $G$ with transition probabilities $p(x, y) := \mu(x^{-1}y)$. Our goal is to understand the $n$-th step distribution\[
Z_n = g_1 \cdots g_n
\] where $g_i$ are independent random variables distributed according to $\mu$.

There are several generalizations of Bellman, Furstenberg and Kesten's theory of non-commuting random walks: random walks on Lie groups (cf. \cite{benoist2016random} and the references therein); random conformal dynamics (\cite{deroin2007random}); subadditive and multiplicative ergodic theorems due to Kingman  \cite{kingman1968subadditive} and Oseledec \cite{oseledec1968multiplicative}, respectively (see their generalizations \cite{karlsson2006lln}, \cite{gouezel2020subadditive} that incorporates random processes on isometries and non-expanding maps on a space) to name a few. In geometric group theory, there is a strong analogy between rank-1 Lie groups and groups with a non-elementary action on a Gromov hyperbolic space $X$ (\cite{maher2018random}). Given a basepoint $o \in X$, the sample path $(Z_{n} o)_{n \geq 0}$ on $X$ tracks a geodesic and the displacement $d(o, Z_{n} o)$ at step $n$ grows like a sum of i.i.d. random variables with positive expectation. From this one can derive a number of consequences, such as exponential bounds on the drift (\cite{boulanger2022large, gouezel2022exponential}), limit laws (\cite{karlsson1999multiplicative, Bjorklund2009CLThyperbolic,goldsborough2021markov,Gouezel2017Analyticity,Horbez2018CLTMCG}), and identification of the Poisson boundary (\cite{maher2018random, kaimanovich2000poisson, chawla2022poisson}). If the $G$-action on $X$ is compatible with the geometry of $G$ in a suitable sense, one can transfer these results on $X$ to $G$. One of the most successful results in this direction is due to Mathieu and Sisto \cite{Mathieu2020deviation}, who proved a central limit theorem for random walks on acylindrically hyperbolic groups. We refer the readers to \cite{osin2016acylindrically} and \cite{behrstock2019hierarchically} for examples of acylindrically hyperbolic groups and hierarchically hyperbolic groups.

Although the notion of acylindrical hyperbolicity captures a wide range of discrete groups,  acylindrical hyperbolicity of a group is not known to be quasi-isometry invariant or even commensurability invariant. This is because there is no known natural way to transfer a group action through a quasi-isometry. To overcome this, the second author proposed a theory for random walks using a group-theoretic property that does not involve hyperbolic actions, namely, possessing a strongly contracting element \cite{choi2022random1}. Nevertheless, this theory is still not invariant under quasi-isometry.

Meanwhile, certain hyperbolic-like properties are known to be quasi-isometry invariant, such as existence of a Morse quasi-geodesic. Hence one can expect that many consequences of hyperbolicity should hold under quasi-isometry invariant assumptions. To address this, Goldsborough and Sisto \cite{goldsborough2021markov} developed a QI-invariant random walk theory for groups. Given a bijective quasi-isometry $f$ from a group $G$ to a group $H$, the pushforward of the random walk from $G$ to $H$ is not necessarily a random walk, but only an inhomogeneous Markov chain. Nonetheless, if one---equivalently both---groups are non-amenable, the induced Markov chain satisfies some sort of irreducibility, which the authors call \emph{tameness}. At this moment, Goldsborough and Sisto require that $G$ acts on a hyperbolic space $X$ and contains what they call a `superlinear-divergent' element $g$, that is, any path must spend a superlinear amount of time to deviate from the axis of $g$ (see section \ref{section:background} for the definition). Goldsborough and Sisto observed that along a random path arising from a tame Markov chain on $G$, some translates of the superlinear-divergent axis are aligned on $X$, and such alignment is also realized on the Cayley graph of $G$. As a consequence, they  established a central limit theorem for random walks on $H$, which is only quasi-isometric to $G$.

In the setting of Goldsborough and Sisto, still, $G$ is required to possess an action on a hyperbolic space. Our purpose is to remove this assumption and establish a central limit theorem for groups satisfying a QI-invariant property, without referring to a hyperbolic space. 

\begin{theorem}\label{thm:main}
Let $G$ be a finitely generated group with exponential growth, and suppose that $G$ has a superlinear-divergent quasi-geodesic $\gamma: \Z \rightarrow G$. Let $(Z_n)_{n\geq 1}$ be a simple random walk on $G$. Then there exist constants $\lambda, \sigma \geq 0$ such that \[
\frac{d_{X}(o, Z_{n} o)-\lambda n}{\sigma \sqrt{n}} \rightarrow \mathcal{N}(0, 1) \quad \textrm{in distribution}.
\]
\end{theorem}
Note that we only assume existence of a superlinear-divergent quasi-geodesic, as opposed to a superlinear-divergent element. This makes our setting invariant under quasi-isometry; see Lemma \ref{lem:qiInvariance}. In addition, our proof only uses the classical theory of random walks and does not refer to tame Markov chains. 

This theorem applies to groups that are not flat but not of rank 1 either. For example, we can construct a superlinear-divergent element in any right-angled Coxeter group (RACG) that contains a periodic geodesic with geodesic divergence at least $r^3$:
\begin{prop}\label{prop:RACG}
    Let $W_{\Gamma}$ be a Right-angled Coxeter group of thickness $k\geq 2$. Then any Cayley graph of $\Gamma$ contains a periodic geodesic $\sigma$ which is $(f, \theta)$--divergent for some $\theta>0$ and $f(r) \gtrsim r^k$. In particular, simple random walks on $W_\Gamma$ satisfy the central limit theorem. 
\end{prop} 

By $f \gtrsim r^k$ we mean that $f\geq cr^k$ for some sufficiently small $c > 0$. The proof of this lemma is Appendix \ref{appendix:racg}. Such RACGs can be produced following the method in 
\cite{Levcovitz2022thick}, and \cite{Behrstock2017random} shows that there is an abundance of such groups.

Lastly, let us mention the relationship between superlinear-divergence and the strongly contracting property, which is a core ingredient of the second author's previous work \cite{choi2022random1}. In general, a superlinear-divergent axis need not be strongly contracting and vice versa. Hence, the present theory and the theory in \cite{choi2022random1} are logically independent. We elaborate this independence in Appendix \ref{appendix:contracting}.

\subsection*{Outline of the paper}

Our main idea is to bring Gou{\"e}zel's recent theory of pivotal time construction for random walks \cite{gouezel2022exponential}. Here, the key ingredient is a local-to-global principle for alignments between quasigeodesics. Lacking Gromov hyperbolicity of the ambient group, we establish such a principle among sufficiently long superlinear-divergent geodesics (Proposition \ref{prop:induction}). For this purpose, in Section \ref{section:background} we continue to develop the theory of superlinear-divergent sets after Goldsborough and Sisto \cite{goldsborough2021markov}. In Section \ref{section:alignment}, we discuss alignment of superlinear-divergent geodesics. In Section \ref{section:probablistic}, we estimate the probability for alignment to happen during a random walk. This yields a deviation inequality (Lemma \ref{lem:deviation}) that leads to the desired central limit theorem.

\subsection*{Acknowledgement}
This project was initiated at the AIM workshop
``Random walks beyond hyperbolic groups", after a lecture by Alex Sisto on his work with Antoine Goldsborough. We would like to thank Alex Sisto, Ilya Gekhtman, S{\'e}bastien Gou{\"e}zel, and Abdul Zalloum for many helpful discussions. We are also grateful to Anders Karlsson for suggesting references and explaining the background.

The first author was partially supported by an NSERC CGS-M grant. 
The second author is supported by Samsung Science \& Technology Foundation (SSTF-BA1702-01 and SSTF-BA1301-51) and by a KIAS Individual Grant (SG091901) via the June E Huh Center for Mathematical Challenges at Korea Institute for Advanced Study. The third author was partially supported by an NSERC CGS-M Grant. The fourth author was partially supported by
NSERC Discovery grant, RGPIN 06486.

\section{Superlinear-Divergence}\label{section:background}
For this section, let $X$ be a geodesic metric space. For points $x,y \in X$, we will use the notation $[x, y]$ to mean an arbitrary geodesic between $x$ and $y$ (note: not unique in general). If $\alpha$ is a quasi-geodesic, and $x,y \in \alpha$, we use $[x, y]|_\alpha $ to denote the quasi-geodesic segment from $x$ to $y$ along $\alpha$. Throughout, all paths are continuous maps from an interval into $X$.

We adopt the definition in \cite{goldsborough2021markov}. For a set $Z \subseteq X$ and constants $A, B > 0$, we say a map $\pi = \pi_{Z} : X \rightarrow Z$ is an $(A, B)$--coarsely Lipschitz projection if
\[
\forall x, y \in X, \quad  d(\pi(x), \pi(y)) \le A d(x, y) + B
\]
and
\[
\forall z \in Z, \quad  d(\pi(z),z) \leq B.
\]
We say that a map $\pi$ is coarsely Lipschitz if it is $(A,B)$-coarsely Lipschitz for some $A,B>0$. Note that a coarsely Lipschitz projection is comparable to the closest point projection: for any $x \in X$ we have 
\begin{align*}
d(x, \pi(x)) 
&\le \inf_{z \in Z} \big( d(x, z) + d(z, \pi(z)) + d(\pi(z), \pi(x))\big) \\
&\le \inf_{z \in Z} \big( d(x, z) + B + (Ad(x,z)+B) \big)  \\
&\le (A+1) d(x, Z) + 2B.
\end{align*}

We say a function $f : \mathbb{R}_{+} \rightarrow \mathbb{R}_{+}$ is \emph{superlinear} if it is concave, increasing, and
\[
\lim_{x \to \infty}\frac{f(x)}{x} = \infty.
\]

\begin{definition}[{cf. \cite[Definition 3.1]{goldsborough2021markov}}]\label{dfn:superdivergent}
Let $Z$ be a closed subset of a geodesic metric space $X$, let $\theta > 0$ and let $f : \mathbb{R}_{+} \rightarrow \mathbb{R}_{+}$ be superlinear. We say that $Z$ is \emph{$(f,\theta)$--divergent} if there exists an $(A, B)$--coarsely Lipschitz projection $\pi_{Z} : X \rightarrow Z$ such that for any $R>0$ and any path $p$ outside of the $R$--neighborhood of $Z$, if the endpoints $p_{-}$ and $p_{+}$ of the path $p$ satisfy  
\[
d(\pi_{Z}(p_{-}),\pi_{Z}(p_{+})) > \theta
\]
then the length of $p$ is at least $f(R)$.

We say that $Z$ is \emph{superlinear-divergent} if it is $(f,\theta)$--divergent for some constant $\theta > 0$ and a superlinear function $f : \mathbb{R}_{+} \rightarrow \mathbb{R}_{+}$.
\end{definition}

The following lemma shows that the existence of a superlinear-divergent quasi-geodesic in a group $G$ is a quasi-isometry invariance.

\begin{lem}\label{lem:qiInvariance}
Let $X$ be a geodesic metric space containing a superlinear-divergent subset $Z$, and let $\phi:X\to Y$ be a quasi-isometry. Then $\phi(Z)$ is also superlinear-divergent.
\end{lem}
\begin{proof}
Let $Z \subset X$ be $(f,\theta)$--divergent with a coarsely Lipschitz projection $\pi_{Z}$.  Let $\phi:X\to Y$ be a $(q,Q)$--quasi-isometry. Then $\pi_Z$ pushes forward to a coarsely Lipschitz projection $\pi_{\phi(Z)} = \phi \circ \pi_Z \circ \phi^{-1}$.

Note that the pullback under $\phi$ of a continuous path in $Y$ may not be a continuous path in $X$. But by the taming quasi-geodesics lemma (Lemma III.H.1.11 of \cite{Bridson1999}), we can find a continuous path within the $(q+Q)$--neighborhood of $\phi^{-1}(p)$ with the same endpoints.  

Fix $R>0$. Suppose $p$ is a path in $Y$ outside of a $R$--neighborhood of $\phi(Z)$, and suppose the endpoints $p_-$ and $p_+$ satisfy
\begin{equation*}
    d(\pi_{\phi(Z)}(p_-),\pi_{\phi(Z)}(p_+))>\theta',
\end{equation*}
where $\theta' = q \theta + Q$. Then let $p'$ be a continuous path in the $(q+Q)$--neighborhood of $\phi^{-1}(p)$ with endpoints $p'_{-} \in \phi^{-1}(p_{-})$and $p'_{+} \in \phi^{-1}(p_+)$. It follows that $p'$ is outside of the $\left(\frac{R}{q}-q-2Q\right)$--neighborhood of $Z$. Moreover, the endpoints have projections bounded by
\begin{equation*}
    d_{Z}(\pi_{Z}(p'_-),\pi(p'_+))>\theta.
\end{equation*}
Superlinear divergence of $Z$ lets us conclude that 
\[
l_{X}(p') > f \left(\frac{R}{q}-q-2Q\right),
\]
 so $l_{Y}(p) > g(d)$ where 
\begin{equation*}
     g(x) = \frac{1}{q}f \left(\frac{x}{q}-q-2Q \right)-Q
\end{equation*}
is a superlinear function.
\end{proof}
\begin{cor}\label{cor:qiInvariance}
Suppose a finitely generated group $G$ contains a superlinear-divergent bi-infinite quasi-geodesic $\gamma: \R \rightarrow G$. Let $H$ be a finitely generated group quasi-isometric to $G$. Then $H$ also contains a superlinear-divergent bi-infinite quasi-geodesic.
\end{cor}

We now establish basic consequences of superlinear divergence of a geodesic. In part, superlinear-divergent geodesics are ``constricting'' (in the sense of \cite{arzhantseva2015growth} and \cite{sisto2018contracting}) up to a logarithmic error. This will be formulated more precisely in Lemma \ref{lem:quickLanding2}.

\begin{lem}\label{lem:halfOrTwice}
For each superlinear function $f$ and positive constants $A, B, K, \theta, q, Q$, there exists a constant $K_{0}>1$ such that the following holds.

Let $Z$ be an $(f, \theta)$--divergent subset of $X$ with respect to an $(A, B)$--coarsely Lipschitz projection $\pi_{Z}$, and let $\alpha : [0, M] \rightarrow X$ be a geodesic in $X$ such that 
\[
d\big(\pi_{Z}\alpha(0), \pi_{Z}\alpha(M)\big) \ge \theta
\qquad\text{and}\qquad 
d(\alpha(0), Z) > K_{0}.
\]
Then there exists $t \in [0,M]$ such that 
\[
d\big(\pi_{Z}\alpha(0), \pi_{Z}\alpha(t)\big) \le \theta+B,
\]
and either
\[
d(\alpha(t), Z) \ge K \cdot d(\alpha(0), Z) 
\qquad\text{or}\qquad
d(\alpha(t), Z) \le \frac 1K \cdot d(\alpha(0), Z).
\]
\end{lem}

\begin{proof}
Let $A,B$ be the coarsely Lipschitz constants of $\pi_Z$. Choose $K' >1$ large enough such that for all $t > K'$,
\[
\frac{f(t)}{t} \ge K(K+5B+\theta+1)(A+1).
\]
Let 
\[
\tau := \inf \big\{t \in [0,M]: d(\pi_{Z}\alpha(0), \pi_{Z}\alpha(t)) \ge \theta\big\}.
\]
By the $(A, B)$-coarse Lipschitzness of $\pi_{Z}$, we have
\[
d(\pi_{Z}\alpha(0), \pi_{Z}\alpha(t)) \le \theta+B
\]
for all $t \in [0, \tau]$. We now take $K_{0} = K'K$. For convenience, let $d_{t} := d(\alpha(t), Z)$ for each $t$. The desired conclusion holds if $d_{t} \le d_{0}/K = K'$ for some $t \in [0, \tau]$; suppose not. Under this assumption, we show that $d_\tau > Kd_0$. By superlinear-divergence of $Z$,
\[
l(\alpha([0, \tau])) 
\ge f\left( \frac{d_{0}}{K}\right) 
\ge K(K+5B+\theta+1)(A+1) \cdot \left(\frac{d_{0}}{K}\right)
\ge (K+5B+\theta+1)(A+1) d_{0}.
\]
Note also, since $\alpha$ is a geodesic,
\[\begin{aligned}
l(\alpha([0,\tau]))
&\le d(\alpha(0), \pi_{Z}(\alpha(0))) + d(\pi_{Z}(\alpha(0)), \pi_{Z}(\alpha(\tau))) + d(\pi_{Z}(\alpha(\tau)), \alpha(\tau))\\
&\le ((A+1) d_{0} + 2B) + (\theta+B) + [(A+1) d_{\tau} + 2B].
\end{aligned}
\]
Combining these, we have 
\begin{align*}
d_{\tau} &\ge \frac{1}{A+1}[(K+5B+\theta+1)(A+1) d_{0} -5B-(A+1)d_0 - \theta]\\
&\ge K d_{0} + (5B+ \theta) \left(d_0-\frac{1}{A+1} \right)\\
&\ge K d_{0},
\end{align*}
where the final inequality is due to $d_{0} \ge K_{0} = KK' > 1 > 1/(A+1)$.
\end{proof}

The following lemma is a technical calculation that will be used in the proof of Lemma \ref{lem:quickLanding2} to examine the behaviour of a sequence of points along a geodesic whose projections are making steady progress.
\begin{lem}\label{lem:3points}
Let $\pi_Z : X \rightarrow Z$ be an $(A, B)$--coarsely Lipschitz projection onto a subset $Z$ of $X$ and let $K>0$. Suppose $x, z \in X$, $y \in [x, z]$ satisty \[
\begin{aligned}
d\big(\pi_{Z}(x), \pi_{Z}(y)\big) &<K, \\
d\big(\pi_{Z}(y), \pi_{Z}(z)\big) &< K, \text{and} \\
d(x, Z),\, d(z, Z) &\le \frac{1}{8(A+1)} d(y, Z).
\end{aligned}
\]
Then we have $d(y, Z) \le 2K + 16B$.
\end{lem}

\begin{proof}
Suppose to the contrary that $d(y, Z) > 2K + 16B$. First, the assumption tells us that \[
(A+1) d(x, Z) + 2B \le \frac{1}{8} d(y, Z) + 2B \le \frac{1}{4} d(y, Z).
\] 
This forces \[
d(x, y) \ge d(y, \pi_{Z}(x)) - d(x, \pi_{Z}(x)) \ge d(y, Z) - (A+1)d(x, Z) - 2B \ge \frac{3}{4} d(y, Z)
\]
and similarly $d(y, z) \ge \frac{3}{4} d(y, Z)$, which leads to 
\[d(x, z) = d(x, y) + d(y, z) \ge  \frac{3}{2}d(y, Z).\] 
Meanwhile, note that \[
d(x, z) \le d(x, \pi_{Z}(x)) + d(\pi_{Z}(x), \pi_{Z}(z)) + d(\pi_{Z}(z), z) \le  \frac{1}{2}d(y, Z) + 2K.
\]
Hence, we have \[\frac{3}{2}d(y, Z) \le \frac{1}{2}d(y, Z) + 2K,\] which contradicts the assumption.
\end{proof}

The following is the main lemma. 

\begin{lem}\label{lem:quickLanding2}
Let $Z$ be an $(f, \theta)$--divergent subset of $X$. Then, for any $\delta >0$, there exists $K_1 > 0$ such that the following holds. For any $x, y \in X$, if 
\[
d(\pi_{Z}(x), \pi_{Z}(y)) > \delta (\log d(x, Z) + \log d(y, Z)) + K_{1},
\]
 then there exist a subsegment $[p_{x}, p_{y}]$ of $[x, y]$ and points $q_{x}, q_{y} \in Z$ such that: \begin{enumerate}
    \item $d(p_x, q_x), d(p_y, q_y) < K_1$
    \item $d(q_{x}, \pi_{Z}(x)) \le \delta \log d(x, Z) + K_{1}$,
    \item $d(q_{y}, \pi_{Z}(y)) \le \delta \log d(y, Z) + K_{1}$,
    \item the segment $[p_x, p_y]$ is in the $K_{1}$--neighbourhood of $Z$.
\end{enumerate}
\end{lem}
Roughly speaking, parts (1), (2) and (3) state that the geodesic $[x,y]$ will enter the 
$K_1$--neighbourhood of $Z$ exponentially quickly from both sides and part (4) states that it stays near $Z$ in the middle 
(See Figure~\ref{fig:quicklanding2}). 

\begin{figure}
    \centering
    \includegraphics[width = \linewidth]{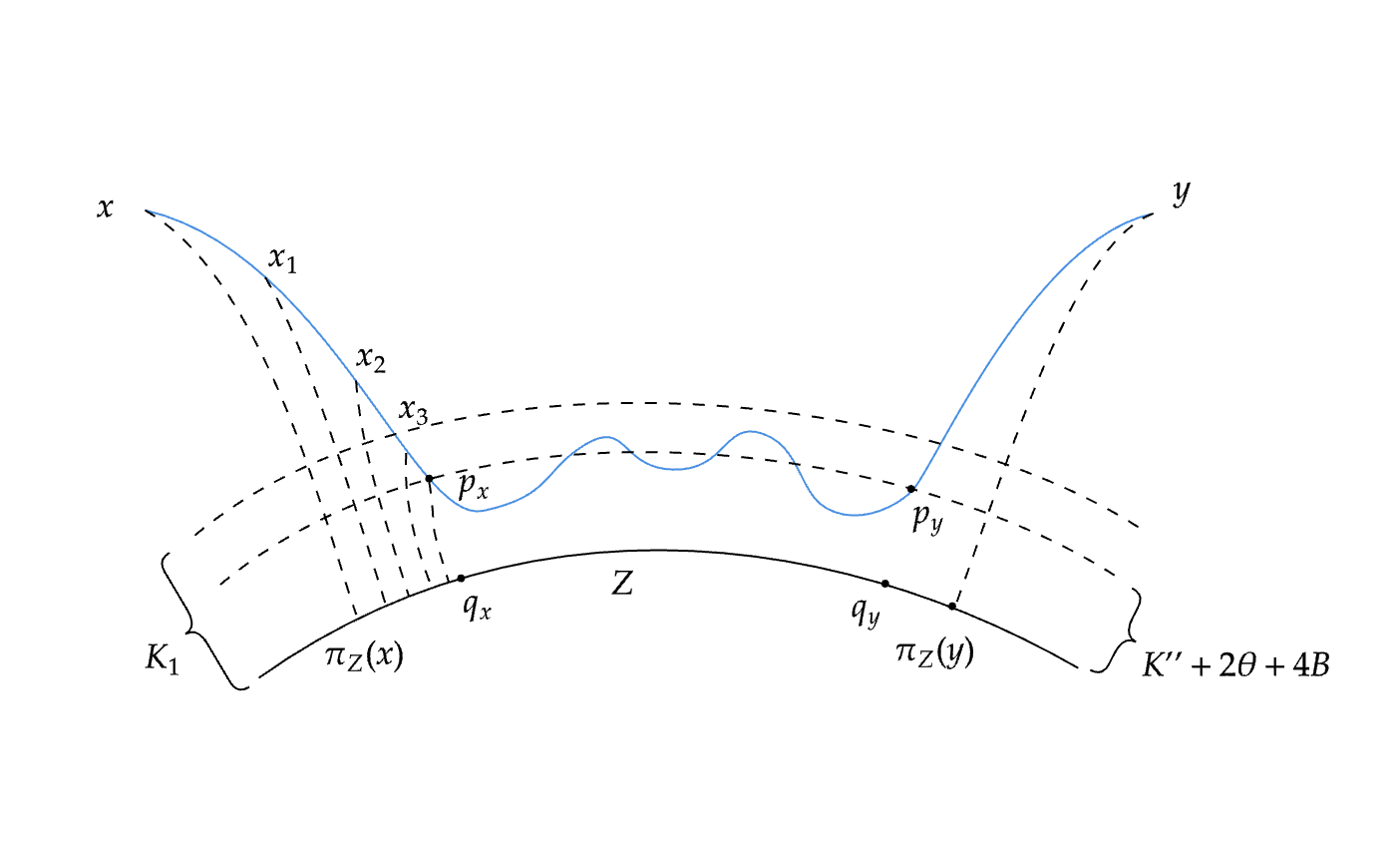}
    \caption{A geodesics whose endpoints project sufficiently far apart onto a superlinear-divergent set $Z$ must enter and exit a small neighborhood of $Z$ near the projections.}
    \label{fig:quicklanding2}
\end{figure}

\begin{proof}
Let $(A, B)$ be the coarsely Lipschitz constants for $\pi_{Z}$. Let $K' = 8(A+1) + \operatorname{exp}(\frac{\theta + B}{\delta})$, let $K'' = K_{0}(K')$ be as in Lemma \ref{lem:halfOrTwice}, and let \[
K_{1} = (2A+3)(K'' + 2\theta + 4B) + 5B + \theta + \log K'.
\]
If $[x, y]$ entirely lies in the $(K'' + 2\theta + 4B)$--neighborhood of $Z$, we can take $p_{x} = x$, $p_{y} = y$, $q_{x} = \pi_{Z}(x)$ and $q_{y} = \pi_{Z}(y)$.

If not, we analyze the subsegments of $[x,y]$ outside of the $(K'' + 2\theta + 4B)$--neighborhood of $Z$. Let $[x', y']$ be an arbitrary connected component of 
\[
[x, y] \setminus N_{K'' + 2\theta + 4B} (Z) := \big\{ p \in [x, y]: d(p, Z) \ge K'' + 2\theta + 4B\big\}.
\]
We will take a sequence of points $\{x_{i}\}_{i=0,\ldots,M}$ on $[x', y']$, associated with a sequence of real numbers $\{r_{i} := d(x_i,Z)\}_{i=0,\ldots,M}$ (Figure \ref{fig:quicklanding2}). We construct the sequence recursively. Start by choosing $x_{0} := x'$, then recursively choose $x_{i+1} \in [x_{i}, y']$ such that 
\[
d\big(\pi_{Z}(x_{i}), \,\pi_{Z}(x_{i+1})\big) \le \theta + B.
\]
and either
\[
r_{i+1} \ge K'r_{i} \qquad\text{or}\qquad r_{i+1} \le r_{i}/K'.
\]
Such $x_{i+1}$ must exist when $d(\pi_{Z}(x_{i}), \pi_{Z}(y')) \geq \theta + B$, due to Lemma \ref{lem:halfOrTwice}. The process terminates at step $M$ when 
\[
d(\pi_{Z}(x_M), \pi_{Z}(y')) \le \theta + B
\qquad\text{or}\qquad
r_{M} \le K''+ 2\theta + 4B.
\]
We first observe that, by Lemma \ref{lem:3points}, for any $i$, we cannot simultaneously have 
\[ 
r_{i} \ge K'r_{i-1} \qquad \text{and}\qquad r_i \geq K'r_{i+1}.
\]
Hence, the only possibilities for the sequence is either: \begin{enumerate}[label=(\roman*)]
\item $r_{i}$ keeps decreasing, 
\item $r_{i}$ keeps increasing, or
\item $r_{i}$ decreases at first and then keeps increasing.
\end{enumerate}
We will apply this observation in two cases depending on the endpoints of $[x',y']$.  

\textbf{Case 1.} One (or both) of the endpoints is $x$ or $y$. 

WLOG, consider the case $x' = x$. We will show that these segments enter the $K_1$--neighbourhood of $Z$ exponentially quickly. Then we will choose $p_x$ to be the entrance point. 

We first see that the sequence will not persist until $y$. Choose index $j$ such that
\[ r_{j} = \min_{i = 0, \ldots, M} r_{i}.\]
If the minimum satisfies
\[ r_{j} > K''+2\theta + 4B, \]
 then the sequence persists until $y$. In this case, the sequence decreases until the minimum, then keeps increasing until the end. It terminates when
\[ 
d(\pi_{Z}(x_{M}), \pi_{Z}(y)) \le \theta +  B,
\]
and the index satisfies
\[
M \ge \frac{1}{\theta + B}d(\pi_{Z}(x), \pi_{Z}(y)) - 1.
\]
Moreover,
\[
 \frac{r_{M}}{r_{j}} \cdot \frac{r_{0}}{r_{j}} \ge K'^{M}.
\]
Combining the three inequalities above we have
\[
\log d(y, Z) + \log d(x, Z) - 2 \log r_{j} \ge d(\pi_{Z} (x), \pi_{Z}(y)) \cdot \frac{\log K_{2} }{\theta+B} - \log K'.
\]
Hence, \[
d\big(\pi_{Z}(x), \,\pi_{Z}(y)\big) \le \delta \big(\log d(x, Z) + \log d(y, Z)\big) + \log K',
\]
a contradiction. Hence, the sequence $\{r_{i}\}_{i}$ keeps decreasing as in Figure \ref{fig:quicklanding2}, and it terminates when 
\[
r_{M} \le K'' + 2\theta + 4B.
\]
Choose $p_x = x_{M}$ and take $q_{x} \in Z$ such that 
\[
d(p_{x}, q_{x}) \le K'' + 2\theta + 4B.
\]
This choice of $p_x$ and $q_x$ guarantees that 
\begin{align*}
d(\pi_{Z}(p_{x}), q_{x}) &\le d(\pi_{Z}(p_{x}), \pi_{Z}(q_{x})) + d(\pi_{Z}(q_{x}), q_{x}) \\
&\le \big(A d(p_{x}, q_{x}) + B \big) + B \\
&\le K_{1}.
\end{align*}
Moreover, we have 
\[
K'^{M} \le r_{0}/r_{M},
\]
which implies
\[M \le \frac{\log r_{0} - \log r_{M}}{\log K'}.
\]
So 
\[
d(\pi_{Z}(p_x), \pi_{Z}(x)) \le (\theta  + B) M \le \delta \log d(x, Z),
\] 
and consequently
\[
d(q_{x}, \pi_{Z}(x)) \le \delta \log d(x, Z) + K_{1}
\]
as desired. We may apply the same argument to choose $p_y \in [x,y]$  and $q_{y} \in Z$  such that 
\[ d(p_{y}, q_{y}) \le K_{1} \qquad \text{and} \qquad d(q_{y}, \pi_{Z}(y)) \le \delta \log d(y, Z) + K_{1}.\]

\textbf{Case 2.} The endpoints $x'$ and $y'$ both belong to the closure of $N_{K'' + 2\theta + 4B}(Z)$. 

These are segments between our choice of $p_x$ and $p_y$. We show that they are within the $K_1$--neighbourhood of $Z$.

In this case,
\[
d(x', Z) = d(y', Z) = K'' + 2\theta + 4B \le d(p, Z) \quad \big(\forall p \in [x', y']\big).
\]
Observe that $r_{i}$ cannot decrease as first since $[x',y']$ lies outside the $(K''+2\theta+4B)$--neighbourhood of $Z$. But $r_{i}$ also cannot keep increasing, because $d(x', Z) = d(y', Z)$. So the process must stop at the very beginning, that is,
\[
M=0 \qquad \text{and} \qquad d(\pi_{Z}(x'), \pi_{Z}(y')) \le \theta + B.
\] 
Then we have 
\begin{align*}
d(x', y') &\le d\big(x', \pi_{Z}(x')\big) + (\theta + B) + d\big(\pi_{Z}(y'), y'\big) \\
&\leq \bigl((A+1)d(x',Z)+2B \bigr) + (\theta + B) + \bigl((A+1)d(y',Z)+2B \bigr) \\
& = 2(A+1)(K'' + 2\theta + 4B) + 4B + (\theta + B).
\end{align*}
From this, we deduce that $[x', y']$ lies in the $K_{1}$--neighborhood of $Z$.
\end{proof}
The next lemma helps us strengthen Lemma \ref{lem:quickLanding2} to a statement about Hausdorff distance. 
\begin{lem}\label{lem:fellowTravelVAR}
Let $K, M, M'$ be positive constants and $\alpha: [0, M] \rightarrow X$ and $\beta : [0, M'] \rightarrow X$ be $(q, Q)$--quasi-geodesics. Suppose that $\alpha$ is contained in a $K$--neighborhood of $\beta$ and \[
d(\alpha(0), \beta(m))< K, \quad d(\alpha(M), \beta(n)) < K
\]hold for some $0 \le m < n \le M'$. Then we have\[
    d_{\rm Haus}(\alpha, \beta|_{[m, n]}) \leq K + Q+6q^{6} Q + 2Kq^{5}.
    \]
\end{lem}

\begin{proof}
Let us define a map $h$ from $[0, M]$ to $[0, M']$. For each $t \in [0, M]$ let $h(t) \in [0, M']$ be such that $d(\alpha(t), \beta(h(t))) \le K$. Without loss of generality, set $s_{0} := m$ and $s_{M} := n$. This map is well-defined, and is a $(q^{2}, K +  2qQ)$--quasi-isometric embedding of $[0, M]$ into $\mathbb{R}$. Indeed, note that
\begin{align*}
|h(t) - h(t')| &\le q d(\beta(h(t)), \beta(h(t'))) + qQ \\
&\le qd(\alpha(t), \alpha(t')) + 2K + qQ \\
&\le q^{2} |t-t'| +  K + 2qQ
\end{align*}
and \begin{align*}
	|t - t'| &\le q d(\alpha(t), \alpha(t')) + qQ \\
	&\le qd(\beta(h(t)), \beta(h(t'))) + 2K + qQ \\
	&\le q^{2} |h(t)-h(t')| +  K + 2qQ.
\end{align*}

From the very definition, it is clear that $\alpha$ and $\beta(h([0, M]))$ are within Hausdorff distance $K$. Next, as $h$ is a QI-embedding of $[0, M]$ into $\mathbb{R}$ that sends $0$ and $M$ to $m$ and $n$, its image $h([0, M])$ is $2qQ$-connected and $h([0, M])$ is contained in \[
[s - 6q^{5}Q - 2K q^{4}, t + 6q^{5}Q + 2K q^{4}].
\]
In particular, $h([0, M])$ and $[m, n]$ are within Hausdorff distance $6q^{5}Q + 2K q^{4}$. By applying $\beta$, we deduce that $\beta(h([0, M]))$ and $\beta|_{[m, n]}$ are within Hausdorff distance $6q^{6} Q + 2Kq^{5} + Q$. Combining all these, we conclude that \[
    d_{\rm Haus}(\alpha, \beta|_{[m, n]}) \leq K + Q+6q^{6} Q + 2Kq^{5}. \qedhere
    \]
\end{proof}

\begin{cor}\label{cor:fellowTravel}
In the setting of Lemma \ref{lem:quickLanding2}, assume that $Z$ is a $(q,Q)$--quasi-geodesic. Then for some constant $K_2$ depending on $f,\theta,q,Q,\delta$,
    \[
    d_{\rm Haus}([p_x,p_y],[q_x,q_y]|_{Z}) \leq K_2.
    \]
\end{cor}

As another corollary of Lemma \ref{lem:quickLanding2}, we can replace a superlinear-divergent quasigeodesic on $X$ with a superlinear-divergent geodesic.

\begin{cor}\label{cor:superDivGeod}
Let $\gamma$ be a bi-infinite $(f, \theta)$--divergent quasigeodesic on a proper space $X$. Then there exists a bi-infinite $(f',\theta')$--divergent geodesic $\gamma'$ such that $d_{Haus}(\gamma, \gamma')$ is finite. Specifically, $f'(x) = f(x-C)$, $\theta' = \theta + 2C$ where $C$ is the constant Hausdorff distance between $\gamma$ and $\gamma'$.
\end{cor}

\begin{proof}
Let $\gamma: \Z \rightarrow X$ be an $(f, \theta)$--divergent $(q, Q)$--quasigeodesic on $X$. Let $K_{1}$ be the constant given by Lemma \ref{lem:quickLanding2} for $Z = \gamma$ and $\delta = 0$. For each sufficiently large $n$, we note that \[
d(\pi_{\gamma}(\gamma(n)), \pi_{\gamma}(\gamma(-n))) \ge d(\gamma(n), \gamma(-n)) - 2B > \frac{2n}{q} - Q - 2B > K_{1}.
\]
Lemma \ref{lem:quickLanding2} tells us that there exists a subsegment $[p_{-n}, p_{n}]$ of $[\gamma(-n), \gamma(n)]$ and $j_{-n}, j_{n} \in \Z$ such that \[
d(p_{-n}, \gamma(j_{-n})) \le K_{1}, \quad d(p_{n}, \gamma(j_{n})) \le K_{1},
\] \[
d(\gamma(j_{-n}), \gamma(-n)) \le d(\gamma(j_{-n}), \pi_{\gamma}(\gamma(-n))) + d(\pi_{\gamma}(\gamma(-n)), \gamma(-n)) \le K_{1}+B,
\]\[
d(\gamma(j_{n}), \gamma(n)) \le d(\gamma(j_{n}), \pi_{\gamma}(\gamma(n))) + d(\pi_{\gamma}(\gamma(n)), \gamma(n)) \le K_{1}+B,
\]
and such that $[p_{-n}, p_{n}] \subseteq N_{K_{1}} (\gamma)$. By Lemma \ref{lem:fellowTravelVAR}, $[p_{-n}, p_{n}]$ and $\gamma([j_{-n}, j_{n}])$ are within Hausdorff distance $K_{1} + Q+6q^{6} Q + 2Kq^{5}$. For simplicity, let $C = K_{1} + Q+6q^{6} Q + 2Kq^{5}$. Note also that \[
j_{-n} < -n + q(K_{1} + B) + Q < 0 < n - q(K_{1} + B) - Q < j_{n}
\]
for large enough $n$. In conclusion, $[p_{-n}, p_{n}]$ contains a point $p$ that is $C$--close to $\gamma(0)$. Moreover, the distance \[d(\gamma(0), p_{n}) > d(\gamma(0), \gamma(j_{n})) - 2K_{1} - B\] grows linearly, and likewise so does $d(\gamma(0), p_{-n})$. Using the properness of $X$ and Arzela-Ascoli, we conclude that the sequence $\{[p_{-n}, p_{n}]\}_{n>1}$ converges to a bi-infinite geodesic $\gamma'$, within a $K_{1}$--neighborhood of $\gamma$. By Lemma \ref{lem:fellowTravelVAR} again, we have $d_{\rm Haus}(\gamma, \gamma') \le C$.

It remains to declare a coarsely Lipschitz projection $\pi_{\gamma'}$ onto $\gamma'$ and show that $\gamma'$ is $(f', \theta')$--divergent with respect to $\pi_{\gamma'}$. Since $d_{\rm Haus}(\gamma,\gamma') \le C$, we can define $\pi_{\gamma'}(z)$ to be a point on $\gamma'$ such that \[d(\pi_{\gamma'}(z), \pi_{\gamma}(z)) < C. \] 
Any path $p$ outside of the $R$--neighborhood of $\gamma'$ is outside of the $(R-C)$--neighborhood of $\gamma$. Moreover, if the endpoints $p_-$ and $ p_+$ of $p$ satisfy that
\[
d\bigl(\pi_{\gamma'}(p_-),\pi_{\gamma'}(p_+)\bigr) > \theta + 2C ,
\]
then by the construction of $\pi_{\gamma'}$,
\[
d\bigl(\pi_{\gamma}(p_-),\pi_{\gamma}(p_+)\bigr) > \theta.
\]
Superlinear divergence of $\gamma$ implies that the length of $p$ is at least $f(R-2C)$. This concludes the proof.
\end{proof}

\subsection{Convention}
From now on, we fix a finitely generated group $G$ with exponential growth which contains a superlinear-divergent bi-infinite geodesic $\gamma : \mathbb{R} \rightarrow G$: this is a QI-invariant property thanks to Corollary \ref{cor:qiInvariance} and Corollary \ref{cor:superDivGeod}.

\section{Alignment}\label{section:alignment}

In this section, we define the alignment of sequences of (subsegments of) superlinear-divergent geodeiscs. The key lemma is Lemma \ref{prop:induction}, which promotes alignment between consecutive pairs to global alignment of a sequence.

\begin{definition}
Given paths $\gamma_{1}, \ldots, \gamma_{N} : \Z \rightarrow G$, integers $m_{i} \leq n_{i}$ and subpaths $\gamma_{i}' := \gamma_{i}([m_{i}, n_{i}])$, we say that $(\gamma_{1}', \ldots, \gamma_{N}')$ is $K$--aligned if: \begin{enumerate}
\item $\pi_{\gamma_{i}}(\gamma_{i-1}')$ lies in $\gamma_{i}\big((-\infty, m_{i} + K]\big)$, and
\item $\pi_{\gamma_{i}}(\gamma_{i+1}')$ lies in $\gamma_{i}\big([n_{i} - K, +\infty)\big)$.
\end{enumerate}
\end{definition}
Note that $\gamma_i$ can be a single point. We will construct linkage words using $K$--aligned paths, starting with the following lemma.

\begin{lem}\label{lem:3segments}
Given a superlinear function $f$, positive constants $\theta, A, B$ and $0<\epsilon, \eta < 0.1$,
there exists a constant $K_{3} = K_{3}(f, \theta, A, B, \epsilon, \eta)$ such that the following holds.

For $i = 1,2$, let $\gamma_i$ be an $(f,\theta)$--divergent geodesic with respect to a $(A,B)$--coarsely Lipschitz projection $\pi_{\gamma_i} : X \to \gamma_i$, and let $\gamma_{i}' = \gamma_{i}([m_{i}, n_{i}])$ be a subpath of $\gamma_{i}$. Let $z \in X$, and let $D > K_3$ be a constant such that: \begin{enumerate}
\item$\diam(\gamma_{1}' \cup \gamma_{2}' \cup z) \leq D$ ;
\item $ |n_{2} - m_{2}| \ge \epsilon \log D$;
\item $(\gamma_1', \gamma_2')$ is $(\eta \epsilon \log D)$--aligned and $(\gamma_2', z)$ is $(2\eta \epsilon \log D)$--aligned.
\end{enumerate}
Then $(\gamma_{1}', z)$ is $(2\eta \epsilon \log D)$--aligned.
\end{lem}
\begin{figure}
    \centering
    \includegraphics[width = \textwidth]{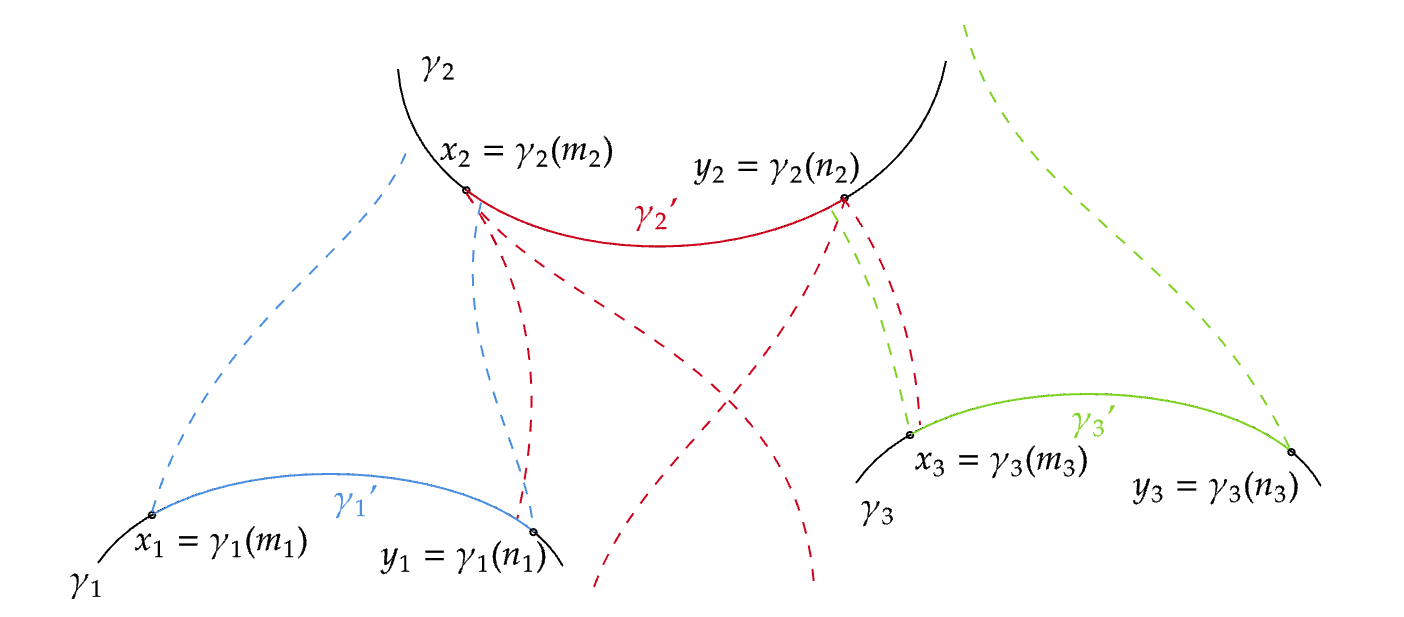}
    \caption{The segments satisfy $(\gamma_1', \gamma_2')$ is $(\eta \epsilon \log n)$--aligned and $(\gamma_2', \gamma_3')$ is $(2\eta \epsilon \log n)$--aligned.}
    \label{fig:aligned}
\end{figure}
\begin{figure}
    \centering
    \includegraphics[width = \textwidth]{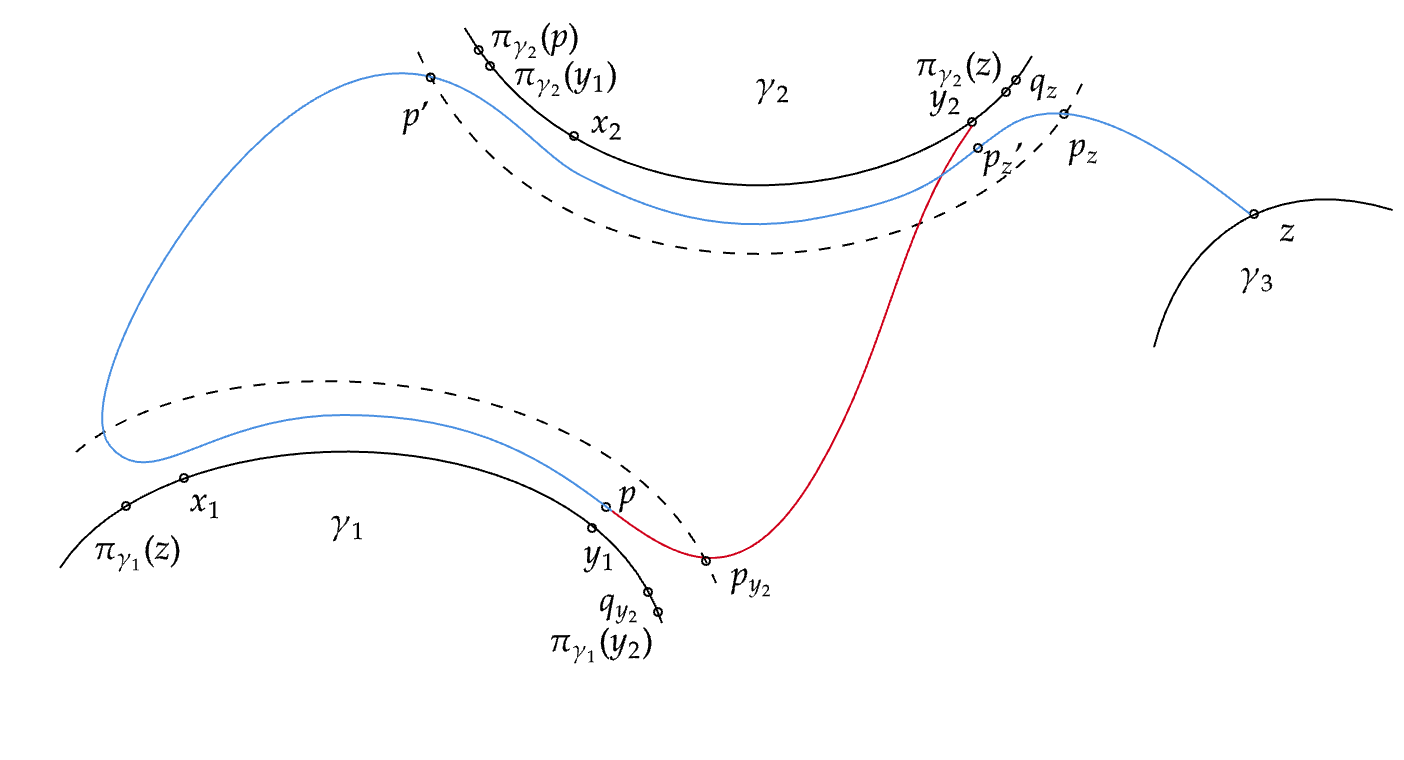}
    \caption{If $\pi_{\gamma_1}(z)$ lies in $ \gamma_1((-\infty, n_1 - 2 \eta \epsilon \log(n)))$, then the geodesic $[y_2,z]$ would fellow travel $\gamma_1'$ then $\gamma_2'$, causing a contradiction.}
    \label{fig:3segments}
\end{figure}
\begin{proof}
We will assume that $D$ is much larger than the constants $K_{1}$ and $K_{2}$ that appears during the argument. For $i=1,2$, denote $x_i = \gamma(m_i)$ and $y_i = \gamma(n_i)$. Suppose for contradiction that $\pi_{\gamma_1}(z)$ lies in $ \gamma_1\big((-\infty, n_1 - 2 \eta \epsilon \log D)\big)$ as in Figure \ref{fig:3segments}. This implies that \[
\begin{aligned}
d(\pi_{\gamma_{1}}(y_{2}), \pi_{\gamma_{1}}(z)) &\ge \eta \epsilon \log D > \frac{\eta \epsilon}{3}( \log d(y_{2}, \gamma_{1}) + \log d(z, \gamma_{1})) + K_1,
\end{aligned}
\]
where $K_{1}$ is the constant given in Lemma \ref{lem:quickLanding2} taking $\delta = \eta \epsilon / 3$. 
By Lemma \ref{lem:quickLanding2}, there exist a subsegment $[p_{1}, p_{2}]$ of $[z, y_{2}]$ and time parameters $s$, $t$ of $\gamma_1$ such that $d(p_{1}, \gamma_{1}(s)) < K_1$, $d(p_{2}, \gamma_{1}(t)) < K_1$ and \[\begin{aligned}
d(\gamma_{1}(s), \pi_{\gamma_{1}}(z)) &< \frac{\eta \epsilon}{3} \log d(z, \gamma_{1}) + K_{1} ,\\
d(\gamma_{1}(t), \pi_{\gamma_{1}}(y_{2})) &< \frac{\eta \epsilon}{3} \log d(y_{2}, \gamma_{1}) + K_{1}.
\end{aligned}
\]

In particular, we have \[\begin{aligned}
s &< \gamma_{1}^{-1} (\pi_{\gamma_{1}}(z)) +  \left( \frac{\eta \epsilon}{3} \log d(y_{2}, \gamma_{1}) + K_{1}\right) \\
&\le n_{1} - \frac{5}{3}\eta \epsilon \log D.
\end{aligned}
\]
A similar calculation shows that $t > n_{1} - \frac{4}{3}\eta \epsilon \log D$. Now let $K_2$ be the constant in Corollary \ref{cor:fellowTravel} so that $\gamma_{1}([s, t])$ and $[p_{1}, p_{2}]$ are within Hausdorff distance $K_{2}$ of each other. In particular, for $p' := \gamma_{1}(n_{1} - 1.5 \eta \epsilon \log D) \in \gamma_{1}([s, t])$, we have a point $p \in [p_{1}, p_{2}] \subseteq [z, y_{2}]$ such that $d(p, p') \le K_{2}$.

Let us now investigate the relationship between $[p, y_{2}]$ and $\gamma_{2}'$. First, the coarse Lipschitzness of $\pi_{\gamma_{2}}$ tells us that \[\begin{aligned}
d(\pi_{\gamma_{2}}(p), \pi_{\gamma_{2}}(y_{2})) &\ge d(\pi_{\gamma_{2}}(p'), y_{2}) - d(\pi_{\gamma_{2}}(p'), \pi_{\gamma_{2}}(p)) - d(y_{2}, \pi_{\gamma_{2}}(y_{2})) \\
&\ge d(\pi_{\gamma_{2}}(p'), y_{2}) - AK_{2}- 2B.
\end{aligned}
\]
Since $\pi_{\gamma_{2}}(p') \in \pi_{\gamma_{2}}(\gamma_{1}')$ is contained in $\gamma_{2}\big( (-\infty, m_{2} + \eta \epsilon \log D) \big)$, we deduce that\[\begin{aligned}
d(\pi_{\gamma_{2}}(p), \pi_{\gamma_{2}}(y_{2})) > (n_{2} - m_{2}) - \eta \epsilon \log D - AK_{2} - 2B > \frac{\eta \epsilon}{3}(\log d(p, \gamma_{2})) + K_{1}.
\end{aligned}
\]
Again, by Lemma \ref{lem:quickLanding2} there exist a subsegment $[p_{1}', p_{2}'] \subseteq [p,z]$ and time parameters $s', t'$ of $\gamma_2$  with $d(p_{1}', \gamma_{2}(s')) < K_{1} $, $d(p_{2}', \gamma_{2}(t')) < K_{1}$ and \[\begin{aligned}
d(\gamma_{2}(s'), \pi_{\gamma_{2}}(p)) &< \frac{\eta \epsilon}{3q} \log d(p, \gamma_{2}) + K_{1},\\
d(\gamma_{2}(t'), \pi_{\gamma_{2}}(y_{2})) &<  K_{1}.
\end{aligned}
\]
This means that \[\begin{aligned}
d(p_{1}', p_{2}') &\ge d(\pi_{\gamma_{2}}(p),\pi_{\gamma_{2}}(y_{2})) 
- d(\gamma_{2}(t'), \pi_{\gamma_{2}}(y_{2}))
-d(\gamma_{2}(s'), \pi_{\gamma_{2}}(p))\\
&\qquad -d(p_{1}', \gamma_{2}(s'))
-d(p_{2}', \gamma_{2}(t')) \\
&\ge  \frac{(n_{2} - m_{2}) - \eta \epsilon \log D}{q} - Q - \frac{1}{3}\eta \epsilon \log D - 4K_{1} \\
&\ge \frac{2}{3}\eta \epsilon \log D.
\end{aligned}
\]
Hence, $[y_{2}, p_{2}']$ is longer than $\frac{2}{3} \eta \epsilon \log D$. But on the other hand,
\[d(p_{2}', y_{2}) \le d(p_{2}', \gamma_{2}(t')) + d(\gamma_{2}(t'), \pi_{\gamma_{2}}(y_{2})) + d(y_{2}, \pi_{\gamma_{2}}(y_{2})) \le 2K_{1}+ B.\] This is a contradiction for sufficiently large $D$.
\end{proof}

\begin{prop}\label{prop:induction}
Let $f$ be a superlinear function, $\theta, A, B > 0$, $0< \epsilon, \eta < 0.1$ and let $K_{3}$ be the constant given in Lemma \ref{lem:3segments}. Let $x, y \in X$, and for $i=1, \ldots, N$, $\gamma_{i}$ be an $(f, \theta)$--divergent geodesic with respect to a $(A, B)$--coarse-Lipschitz projection and let $\gamma_{i}' = \gamma_{i}([m_{i}, n_{i}])$ be a subpath of $\gamma_{i}$. Let $D > K_3$ be a constant such that:
\begin{enumerate}
\item $\diam(\gamma_{1}' \cup \ldots \cup \gamma_{N}') \leq D$;
\item $ |n_{i} - m_{i}| \ge \epsilon \log D$ for each $i$, and 
\item $(x, \gamma_{1}', \ldots, \gamma_{N}', y)$ is $(\eta \epsilon \log D)$--aligned.
\end{enumerate}
Then for each $i$, $(x, \gamma_{i}', y)$ is $(2\eta \epsilon \log D)$--aligned.

\end{prop}
\begin{proof}
This follows inductively from lemma \ref{lem:3segments}. Fixing $i<j$, we show that 
\begin{equation*}
    \pi_{\gamma_i}(\gamma_j') \in \gamma_i((-\infty, m_i+2\eta\epsilon\log D]).
\end{equation*}
If $i = j-1$, immediately by assumption we have 
\begin{equation*}
    \pi_{\gamma_i}(\gamma_j')
    \in \gamma_i((-\infty, m_i+\eta\epsilon\log D])
    \subset \gamma_i((-\infty, m_i+2\eta\epsilon\log D]).
\end{equation*}
Now assuming 
\[
\pi_{\gamma_{i+1}}(\gamma_j) \in \gamma_{i+1}((-\infty, m_{i+1}+2\eta\epsilon\log D]),
\]
since $(\gamma _i,\gamma_{i+1})$ are $(\eta\epsilon \log D)$--aligned, the triple $(\gamma _i,\gamma_{i+1}, \gamma_j)$ satisfies the assumptions in lemma \ref{lem:3segments}. We conclude that 
\[
\pi_{\gamma_{i}}(\gamma_j) \in \gamma_i((-\infty, m_i+2\eta\epsilon\log D]).
\]
Applying the same argument to $\pi_{\gamma_{j}}(\gamma_i)$, $\pi_{\gamma_{j}}(\gamma_k)$, and $\pi_{\gamma_{k}}(\gamma_j)$ shows that
\begin{align*}
\pi_{\gamma_{j}}(\gamma_i) &\in \gamma_j([n_j-2\eta\epsilon\log D, +\infty))\\
\pi_{\gamma_{j}}(\gamma_k) &\in \gamma_j((-\infty, m_j+2\eta\epsilon\log D]) \text{, and}\\
\pi_{\gamma_{k}}(\gamma_j) &\in \gamma_k([n_k-2\eta\epsilon\log D,+\infty)). \qedhere
\end{align*}
\end{proof}

\begin{lem}\label{lem:2segment}
Given a superlinear function $f$, positive constants $\theta, A, B$ and $0<\epsilon, \eta < 0.1$,
there exists constants $K_{4} = K_{4}(f, \theta, A, B, \epsilon, \eta)$ and $C = C(A)$ such that the following holds.

Let $\alpha$ and $\beta$ be $ (f, \theta)$--divergent geodesics with respect to $(A, B)$--coarsely Lipschitz projections. Let $\alpha'$ and $\beta'$ be their subsegments with beginning points $x_1$ and $x_2$, respectively, such that: \begin{enumerate}
    \item $D := \diam(\alpha'\cup \beta') \ge K_{4}$;
    \item $\diam(\alpha') \geq \epsilon \log D$, and
    \item $(\alpha', x_2)$ and $(x_1, \beta')$ are $\eta \epsilon\log D$--aligned.
\end{enumerate}
Then $(\alpha', \beta')$ is $(C \eta\epsilon \log D)$--aligned.
\end{lem}
\begin{proof}
Let $\alpha'= \alpha([m_{1}, n_{1}])$ and $\beta'= \beta([m_{2}, n_{2}])$. Denote $x_1 = \alpha (m_1)$, $y_1 = \alpha (n_1)$, $x_2 = \beta (m_2)$, $y_2 = \beta (n_2)$. Let $C' = 16(A+1)+1$ and $C = (C')^2 + 2$. We first show that 
\[
\pi_{\beta}(\alpha') 
\subset \beta((-\infty, m_2 + C'\epsilon \log D]) 
\subset  \beta((-\infty, m_2 + C\epsilon \log D]).
\]
Suppose to the contrary that for some point $a \in \alpha'$, the projection
\[
\pi_\beta(a) \in \beta([m_2 + C'\epsilon \log D , +\infty)).
\]
Then we have
\begin{align*}
    d(\pi_\beta (x_1), \pi_\beta (a)) 
    & \geq (C'-1) \eta \epsilon \log D\\
    & \geq \frac{1}{16A}(C'-1) \eta \epsilon (\log d(x_1, \beta) + \log d(a,\beta)) + K_1,
\end{align*}
where $K_1 >0$ is the constant as in Lemma \ref{lem:quickLanding2} taking $\delta = \frac{1}{16A}(C'-1)\eta \epsilon$. Then there exists a subsegment $[p_{x_1},p_a]|_\alpha \subset [x_1,a]|_\alpha \subset [x_{1}, y_{1}]|_\alpha$, and points $q_{x_1}, q_a$ on $\beta$ such that 
\begin{align*}
d(p_{x_1}, q_{x_1}), d(p_a, q_a) &< K_1\\
d(q_{x_1}, \pi_{\beta}(x_1)) &\le \left(\frac{1}{16A}(C'-1)\eta \epsilon\right) \log d(x_1, \beta) + K_{1}\\
d(q_{a}, \pi_{\beta}(a)) &\le \left(\frac{1}{16A}(C'-1)\eta \epsilon\right) \log d(a, \beta) + K_{1}.
\end{align*}
Then by Corollary \ref{cor:fellowTravel}, there is a point $p'_{x_1} \in [p_{x_1},p_{a}]|_{\alpha}$ close to $x_2$. The point $p'_{x_1}$ is chosen to be $p_{x_1}$ if $q_{x_1} \in \beta((m_2,\infty))$, or the point where the Hausdorff distance $K_2$ is attained if $q_{x_1} \in \beta((-\infty, m_2])$. The distance is bounded by
\begin{align*}
    d(x_2,p'_{x_1}) 
    &\leq  \max \left(\left(\frac{1}{16A}(C'-1)\eta \epsilon\right) \log d(x_1, \beta) + 2K_{1},K_2 \right)\\
    &\leq \left(\frac{1}{16A}(C'-1) \right)\eta\epsilon \log D + O(1)\\
    &= \left(\frac{A+1}{A} \right)\eta\epsilon \log D + O(1),
\end{align*}
where $K_2$ is the constant in Corollary \ref{cor:fellowTravel}, and $O(1)$ is the implied constant. Projecting to $\alpha$ gives that
\begin{align*}
    d(\pi_\alpha(x_2),p'_{x_1}) \leq  
    d(\pi_\alpha(x_2),\pi_\alpha(p'_{x_1})) + B 
    \leq (A+1)q\eta\epsilon \log D + O(1).
\end{align*}
On the other hand, since $(\alpha', x_{2})$ is $(\eta\epsilon\log D)$--aligned,
\begin{align*}
    d(\pi_\alpha(x_2),p'_{x_1})
    &\geq d(y_1,p'_{x_1}) - \eta\epsilon\log D \\
    &\geq d(p_{a},p'_{x_1}) - \eta\epsilon\log D\\
    &\geq d(x_2,\pi_\beta(a)) - d(x_2,p'_{x_1}) - d(\pi_\beta(a), p_{a})- \eta\epsilon\log D\\
    & \geq \left( \frac{1}{q} C'\eta \epsilon \log D \right) - 2 \left( \frac{C'-1}{16A} \eta \epsilon \log D \right)- \eta\epsilon\log D - O(1)\\
    & \geq  \bigl(14(A+1)-1 \bigr) \eta \epsilon \log D - O(1)
\end{align*}
contradicting the previous inequality when $D$ is sufficiently large.

We now show that 
\[
\pi_{\alpha}(\beta') 
\subset  \alpha((n_1 - C\epsilon \log D,\infty)).
\]
Suppose the contrary that for some point $b \in \beta'$ the projection
\[
\pi_{\alpha}(b) 
\in  \alpha((-\infty, n_1 - C\epsilon \log D)).
\]
We will discuss in two cases. If
\[
\pi_{\alpha}(b) 
\in  \alpha((m_1, n_1 - C\epsilon \log D)) \subset \alpha',
\]
then the previous calculation shows that 
\[
\pi_\beta(\pi_{\alpha}(b)) \in \beta ((-\infty, m_2+C'\eta \epsilon \log D]).
\]
This shows that $(\pi_\alpha(b),[x_2,b]|_\beta, )$ and $(b,[\pi_\alpha(b), y_1]|\alpha)$ are $(C'\eta\epsilon \log D)$--aligned. Moreover, $\diam ([x_2,b]|_\beta \cup [\pi_\alpha(b), y_1]|_\alpha) < D$. So the exact same calculation as before shows that 
\[
\pi_\alpha ([x_2,b]|_\beta) 
\subset \alpha \bigl((- \infty, \pi_\alpha(b)+C'^2 \epsilon \log D )\bigr)
\subset \alpha \bigl((- \infty, n_1-2 \epsilon \log D) \bigr).
\]
This contradicts that 
\[
\pi_\alpha (x_2) \in \alpha \bigl((n_1 - \eta\epsilon\log D , \infty)\bigr).
\]
The remainder case is when $\pi_{\alpha}(b) \in \alpha \bigl((-\infty, m_1)\bigr)$. We will show that this is impossible assuming $\eta < \min \left(\frac{1}{q+2q^2} , \frac{A+2}{A+q+2}\right)$ and $\alpha'$ is long. In this case,
\begin{align*}
    d(\pi_{\alpha}(b),\pi_{\alpha}(x_2)) & \geq \frac{1}{q}(1-\eta)\epsilon \log D\\
    & \geq \frac{1}{(2+A)} \eta \epsilon \log d(b,\alpha) + \log d(x_2,\alpha) - K_1,
\end{align*}
where $K_1$ is the constant in Lemma \ref{lem:quickLanding2} choosing $\delta = \frac{1}{(2+A)} \eta \epsilon$. Then by Lemma \ref{lem:quickLanding2}, there are points $p_{x_2},p_b \in [x_2,b]$ such that
\begin{align*}
    d(p_{x_2}, y_1) & \leq  \frac{1}{2+A} \eta \epsilon \log D +K_1, \text{and}\\
    d(p_b, x_1) &\leq \frac{1}{2+A} \eta \epsilon \log D +K_1.
\end{align*}
Then
\[
    d(\pi_\beta (x_1), p_b) \leq \frac{A}{2+A} \eta \epsilon \log D +O(1).
\]
But on the other hand, $\pi_\beta(x_1) \in \beta((-\infty,m_2 + \eta \epsilon \log D])$ implies that
\begin{align*}
     d(\pi_\beta (x_1), p_b)
     & \geq d(x_2,p_b) - \eta \epsilon \log D \\
     & \geq d(p_{x_2},p_b) - \eta \epsilon \log D \\
     & \geq d(x_1,y_1) - d(x_1,p_b) - d(y_1,p_{x_2}) \eta \epsilon \log D \\
     & \geq \epsilon\log D - \frac{2}{2+A}\eta \epsilon \log D - \eta \epsilon \log D - O(1)\\
     &> \frac{A}{2+A} \eta \epsilon \log n +O(1).
\end{align*}
The last step is due to $\eta < 1/3$. This is a contradiction. 
\end{proof} 

We now construct linkage words. These play the role of Schottky sets in \cite{boulanger2022large, gouezel2022exponential} We use the notation $\mathcal{B}(g, R) := \{h \in G : d(g, h) \le R\}$ to mean the ball of radius $R$ around $g$, and $\mathcal{S}(g, R) := \{h \in G : d(g, h) = R\}$ to mean the sphere of radius $R$ around $g$.

\begin{lem}\label{lem:linkage}
Let $\gamma: \R \rightarrow G$ be a $(f, \theta)$--divergent quasi-geodesic and let $\epsilon>0$. For $K$ sufficiently large, the following holds. For each $m\in \Z$, there exists a subset $S \subseteq G$ with 100 elements such that for each pair of distinct elements $a, b \in S$, we have
\begin{enumerate}
\item $ |a|, |b| = K$ and $|b a^{-1}|, |a^{-1} b|\ge 0.5K$;
\item $\pi_{\gamma}(\gamma(0) a^{-1})$ and $\pi_{\gamma}(\gamma(0)a^{-1} b) \in \mathcal{B}(\gamma(0), \epsilon K)$, and 
\item $\pi_{\gamma}(\gamma(m) a)$ and $\pi_{\gamma}(\gamma(m) ab^{-1}) \in \mathcal{B}(\gamma(m), \epsilon K)$.
\end{enumerate}
\end{lem}

\begin{proof}
Let $K_{0} = K_{0}(0.1\epsilon, f)$ be as in Lemma \ref{lem:quickLanding2}.

Let $\lambda > 1$ be the growth rate of $G$. For $n$ large enough, we have \[
\lambda^{n}\le \# \mathcal{S}(id, n) \le \lambda^{(1+0.1\epsilon)n}.
\]
We consider the sets 
\[\begin{aligned}
    O_1 &:= \left\{g \in \mathcal{S}(id,K): d(\gamma(0), \pi_{\gamma}(\gamma(0)g)) \geq 0.5 \epsilon K\right\}, \\
    O_{2} &:= \{g \in \mathcal{S}(id,K): d(\gamma(m), \pi_{\gamma}(\gamma(m)g)) \ge 0.5\epsilon K \},\\
\end{aligned}\]

We will argue that both of these sets are much smaller than $\mathcal{S}(id,K)$, and use a certain subset of $\mathcal{S}(id,K) \setminus (O_1\cup O_2)$ to construct our set $S$. 

To show that $O_1,O_2$ are relatively small, let us now consider a word $a$ with $|a| = K$ and $d(\pi_{\gamma}(\gamma(0)a^{-1}), \gamma(0)) \ge 0.5\epsilon K$. Then since \[
d(\pi_{\gamma}(\gamma(0) a^{-1}), \pi_{\gamma}(\gamma(0))) \ge 0.5 \epsilon K - B \ge K_{0} + 0.1\epsilon \log B + 0.1\epsilon \log K,
\]
Lemma \ref{lem:quickLanding2} asserts that there exist $p \in [\gamma(0), \gamma(0)a^{-1}]$ and $q \in \gamma$ such that $d(p, q) \le K_{2}$ and $d(p, \pi_{\gamma}(\gamma(0)a^{-1})) \le  \log |a| + K_{1}$. In this case, we have \[\begin{aligned}
d(p, \gamma(0) a^{-1}) &= d(\gamma(0), \gamma(0) a^{-1}) - d(\gamma(0), p) \\
&\le |a| -  d(\gamma(0), \pi_{\gamma}(\gamma(0) a^{-1})) + d(p, \pi_{\gamma}(\gamma(0) a^{-1})) \\
&\le K - 0.5\epsilon K +  \log K + K_{0}.
\end{aligned}
\]
In summary, \[a^{-1} = (\gamma(0)^{-1}q)\cdot(q^{-1} p) \cdot (p^{-1}\gamma(0)a^{-1})\] where, as in figure \ref{fig:decomposition}, \begin{itemize}
\item $\gamma(0)^{-1}q = \gamma(0)^{-1} \gamma(k)$ for some $k$ between $-2qK - Q$ and $2qK+Q$;
\item $|q^{-1}p| \le K_{0}$, and
\item $|p^{-1}\gamma(0)a^{-1}| \le (1 - 0.5\epsilon) K+ \log(1.5K) + K_{0}$.
\end{itemize}

\begin{figure}
    \centering
    \includegraphics[width = \textwidth]{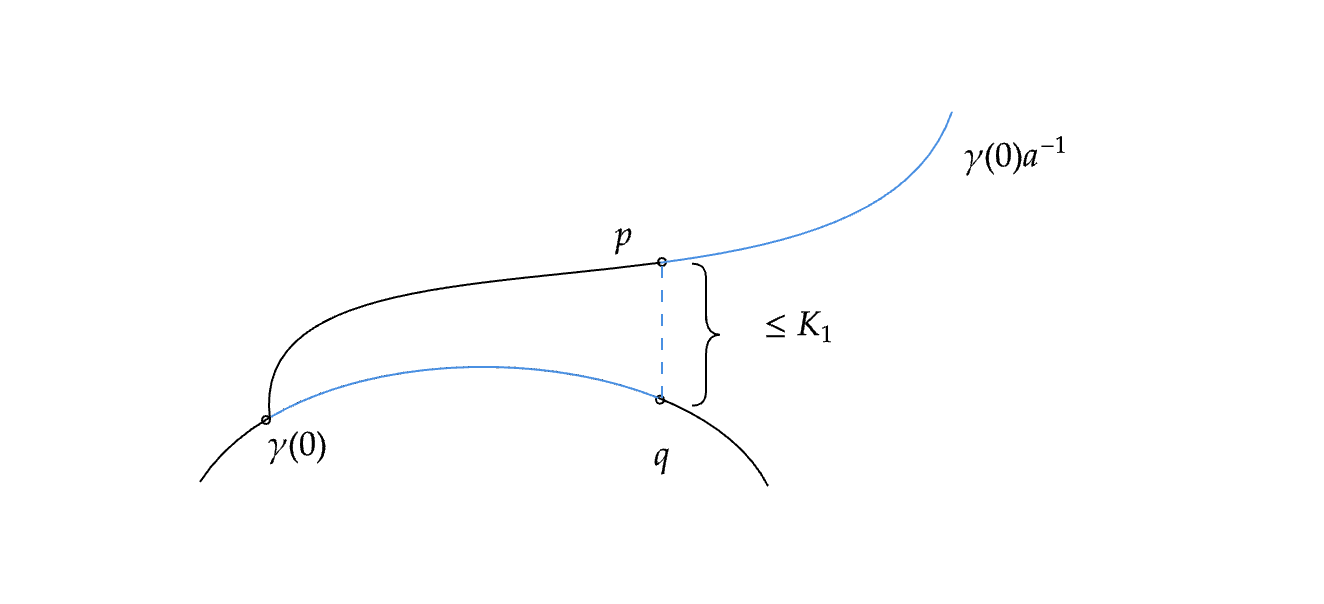}
    \caption{Decomposing $ a ^{-1} $ as a concatenation of well-controlled paths.}
    \label{fig:decomposition}
\end{figure}

For large enough $K$, the number of such elements is at most \[5QK \cdot \lambda^{(1+0.1\epsilon) (1-0.4\epsilon)K} \le 5QK \lambda^{(1-0.3 \epsilon)K}\].

Hence, the cardinality of \[
\mathcal{A} := \{ (g_{1}, \ldots, g_{100}) \in \mathcal{S}(id,K)^{100} : g_i \in O_1 \,\textrm{for some} \,\, i \in [1,100]\}
\]
is at most $100 \cdot (\# S_{0})^{99} \cdot 3QK \lambda^{ (1+0.3) \epsilon K}$---we pick some index $i$ which satisfies the given condition and draw the rest of the elements from $\mathcal{S}(id,K)$. This is exponentially small compared to $\left(\# \mathcal{S}(id,K)\right)^{100}$. 

By a similar logic,\[
\mathcal{B} := \{ (g_{1}, \ldots, g_{100}) \in \mathcal{S}(id,K)^{100} : g_i \in O_2\,\textrm{for some} \,\, i \in [1,100]\}
\] is exponentially small compared to $\mathcal{S}(id,K)^{100}$.

Finally, we observe that for each $h \in \mathcal{S}(id,K)$, there are at most 
\[\#\mathcal{B}(h,0.5K)\leq \lambda^{(1+\epsilon)0.5K}\]
elements $g$ such that $|g^{-1}h| \geq 0.5K$. 

Hence, we deduce that the cardinality of 
\[\mathcal{C} := \left\{ (g_{1}, \ldots, g_{100}) \in S_{0}^{100} : d(g_i,g_j) \leq 0.5K \textrm{for some} \,\,  i \neq j\right\}
\]
is at most $100 \cdot 99 \cdot 2 \cdot \lambda^{0.6K} \cdot (\#\mathcal{S}(id,K)^{99}$, which is exponentially small compared to $\left(\#\mathcal{S}(id,K)\right)^{100}$. Given these estimates, we conclude that for sufficiently large $K$, \[
\mathcal{S}(id,K)^{100} \setminus\big( \mathcal{A} \cup \mathcal{B} \cup \mathcal{C} \big)
\]
is nonempty.

Letting $(g_{1}, \ldots, g_{100})$ be one of its elements, we claim that the choice $S = \{g_i, i = 1,...,100\}$ satisfies the conditions of the lemma.

Note in particular that $g_i^{-1}g_j\neq id$ since its norm is at least $0.5K$. We observe that: \begin{enumerate}
\item \(g_{i}\)'s are all distinct;
\item \(|g_{i}| = K\) for all \(1 \le i\le 100\);
\item \(|g_{i} g_{j}^{-1}|\), \(|g_{i}^{-1} g_{j}|\) \(\ge 0.5K\) for each \(i\neq j\);
\item \(\pi_{\gamma}(\gamma(0) g_{i}^{-1}) \in \mathcal{B}(\gamma(0), 0.5\epsilon K)\) and \(\pi_{\gamma}(\gamma(m) g_{i}) \in \mathcal{B}(\gamma(m), 0.5\epsilon K)\) for each \(1 \le i \le 100\).
\end{enumerate}

It remains to show that \(d(\gamma(0), \pi_{\gamma}(\gamma(0) g_{i}^{-1} g_{j})) < \epsilon K\) for each \(i \neq j\). Suppose not; then for large enough $K$ we have \[
\begin{aligned}
d(\pi_{\gamma}(\gamma(0) g_{i}^{-1}), \pi_{\gamma}(\gamma(0) g_{i}^{-1}g_{j})) &\ge \epsilon K - 0.5 \epsilon K \\
&> 2 \epsilon \log K  > \epsilon \log |g_{i}| + \epsilon \log (|g_{i}|+|g_{j}|) + K_{0} \\
&\ge \epsilon \log d( \gamma, \gamma(0) g_{i}^{-1}) + \epsilon \log d( \gamma, \gamma(0) g_{i}^{-1} g_{j}) + K_{0} \\
\end{aligned}
\]
By Lemma \ref{lem:quickLanding2}, there exists \(p \in [\gamma(0) g_{i}^{-1}, \gamma(0) g_{i}^{-1} g_{j}]\) such that \[
d( p, \pi_{\gamma}(\gamma(0) g_{i}^{-1})) < \epsilon \log d(\gamma, \gamma(0) g_{i}^{-1}) + 2K_{2} \le \epsilon \log K + 2K_{0},
\]
and \(d(\gamma(0), p) < 0.6 \epsilon K\). Here, we have \[
d(p, \gamma(0) g_{i}^{-1}) \ge d(\gamma(0), \gamma(0) g_{i}^{-1}) - d(\gamma(0), p) \ge K - 0.6 \epsilon K,
\]
and \[
\begin{aligned}
d(\gamma(0), \gamma(0) g_{i}^{-1} g_{j}) &\le d(\gamma(0), p) + d(p, \gamma(0) g_{i}^{-1}(g_{j}) \\
&\le 0.6 \epsilon K + [d(\gamma(0) g_{i}^{-1}, \gamma(0) g_{i}^{-1} g_{j}) - d(\gamma(0) g_{i}^{-1}, p)] \\
&\le 0.6 \epsilon K + 0.6 \epsilon K.
\end{aligned}
\]
But this contradicts \(|g_{i}^{-1} g_{j}| \ge 0.5 K\).
\end{proof}

Given a translate of $\gamma$, we can naturally define the projection 
\[\pi_{g\gamma}(x) := g\pi_{\gamma}(g^{-1}x).\]

Since $G$ acts by isometries, this is an $(A,B)$--coarse Lipschitz projection so long as $\pi_{\gamma}$ is as well. The following lemma describes projections between translates of superlinear-divergent quasi-geodesics. 

\begin{lem}\label{lem:1segment}
Let $\alpha$ and $\beta$ be $(f, \theta)$--divergent quasi-geodesics and let $0< \epsilon < \frac{1}{10(A+1)}$. Then there exists $K_{6} > 0$ such that the following holds. Suppose $a \in G$ and $i \in \Z$ satisfy that \begin{enumerate}[label=(\roman*)]
\item $|a| > K_{6}$;
\item $\pi_{\beta}(\beta(0) a^{-1}) \in \mathcal{B}(\beta(0), \epsilon |a|)$; and 
\item$\pi_{\alpha}(\alpha(i) a) \in \mathcal{B}(\alpha(i), \epsilon |a|)$.
\end{enumerate}
Then for each $j \in \Z$, $\pi_{\alpha}(\alpha(i)a \beta(0)^{-1}\beta(j))$ is within distance $\epsilon \log |j| + 2|a|$ from $\alpha(i)$.
\end{lem}
\begin{proof}

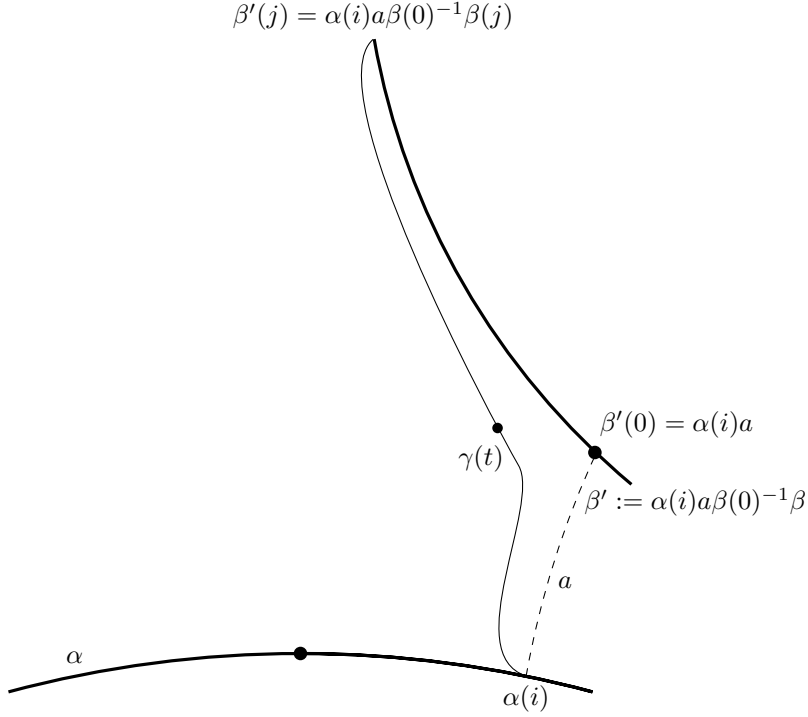
\begin{figure}[ht]
\centering
\begin{tikzpicture}
\def\c{1}
\draw[very thick] (0, 0) arc (90:75:15*\c) arc (75:105:15*\c);
\draw[very thick, shift={(4.4*\c, 2.25*\c)}] (0, 0) arc (-130:-170:10*\c);
\draw[dashed] (3*\c, -0.3*\c) arc (170:155:12*\c); 
\draw (0.965*\c, 8.17*\c) .. controls (0.2*\c, 7.5*\c) and (2.5*\c, 3.2*\c) .. (2.9*\c, 2.5*\c) .. controls (3.2*\c, 1.975*\c) and (2.1*\c, 0) .. (3*\c, -0.3*\c);

\fill (0, 0) circle (0.09);
\fill (3.914*\c, 2.674*\c) circle (0.09);
\fill (2.62*\c, 3*\c) circle (0.07);

\draw (3*\c, -0.3*\c) node[below] {$\alpha(i)$};
\draw (5*\c, 2.75*\c) node[above] {$\beta'(0) = \alpha(i) a$};
\draw (0.965*\c, 8.17*\c) node[above] {$\beta'(j) = \alpha(i) a \beta(0)^{-1} \beta(j)$};
\draw (2.4*\c, 2.9*\c) node[below] {$\gamma(t)$};
\draw (3.52*\c, 0.95*\c) node {$a$};
\draw (-3*\c, -0.05*\c) node {$\alpha$};
\draw (5.25*\c, 2.02*\c) node {$\beta' := \alpha(i) a \beta(0)^{-1} \beta$};

\end{tikzpicture}
\caption{The geodesic $[\alpha(i), \beta'(j)]$ comes near $\beta'(0)$.}
\label{fig:1segment}
\end{figure}
For simplicity, we denote and parameterize the translate of $\beta$
\[
\beta' (j) = \alpha(i) a \beta(0)^{-1} \beta (j).
\]
Let $\gamma: [0, M] \rightarrow G$ be a geodesic connecting $\alpha(i)$ and $\beta'(j)$, see Figure \ref{fig:1segment}. The projection of $\alpha(i)$ onto $\beta'$ is near $\alpha(i) a$:
\[
d\big(\alpha(i) a, \pi_{\beta'} (\alpha(i)) \big) = d \big(\beta(0), \pi_{\beta}(\beta(0)a^{-1}) \big) \le \epsilon |a|.
\]

Then there exists $t \in [0,M]$ such that $\gamma(t) \in  \mathcal{B}(\alpha(i) a, 2 \epsilon |a|)$. If $d(\beta'(j),\alpha(i)a) < 2 \epsilon |a|$, simply take $t=M$ so that $\gamma(t) = \beta'(j)$. And if  $d(\beta'(j),\alpha(i)a) \geq 2 \epsilon |a|$, we obtain such $t$ by applying Lemma \ref{lem:quickLanding2}. Notice
\begin{align*}
     & d(\pi_{\beta'}(\alpha(i)a),\pi_{\beta'}(\beta'(j)))\\
\geq & d(\alpha(i)a,\beta'(j)) - d(\alpha(i)a,\pi_{\beta'}(\alpha(i)a)) - d(\pi_{\beta'}(\beta'(j)),\beta'(j)))\\
\geq & 2 \epsilon |a| - \epsilon |a| - B\\
\geq & \epsilon (\log d(\beta'(j),\beta') + \log d(\alpha(i),\beta')) + K_1,
\end{align*}
where $K_1$ is the constant from Lemma \ref{lem:quickLanding2} taking $\delta = \epsilon$. The last inequality holds when $|a|$ is sufficiently large. Then Lemma \ref{lem:quickLanding2} implies that for some $t$,
\begin{align*}
    d(\gamma(t), \alpha(i)a) 
    & \leq d(\gamma (t), \pi_{\beta'}(\alpha(i))) + d(\pi_{\beta'}(\alpha(i)),\alpha(i)a) \\
    & \leq (\epsilon \log |a| + K_1) + \epsilon |a| \\
    & \leq 2 \epsilon |a|
\end{align*}
for sufficiently large $|a|$. 
Note that \[ t = d(\gamma(0),\gamma(t)) \le d(\alpha(i), \alpha(i) a) + 2\epsilon |a| \le \frac{3}{2} |a|. \]
Now if 
\begin{align*}
d(\pi_{\alpha}(\gamma(0)), \pi_{\alpha}(\gamma(M))) & > \epsilon \log (|j|+|a|)\\
& \geq \epsilon (\log d(\gamma(0),\alpha)+\log d(\gamma(M), \alpha)) - K_1,
\end{align*}
where $K_1$ is the constant from \ref{lem:quickLanding2} taking $\delta = \epsilon$. Then there exists $\tau \in [0, M]$ such that 
\[
d(\gamma(\tau), \pi_{\alpha}(\gamma(M))) \le \epsilon \log (|j| + |a|) + K_{1}
\]
and $\gamma|_{[0, \tau]}$ is contained in the  $K_{1}$--neighborhood of $\alpha$. Notice that $\tau$ cannot be larger than $t$, otherwise $\gamma(t)$ is $K_{1}$--close to $\alpha$; let $p \in \alpha$ be the point such that $d(\gamma(t), p) \le K_{1}$. Then when $|a|$ is sufficiently large,
\[
\begin{aligned}
d(\alpha(i), \pi_{\alpha}(\alpha(i) a)) &\ge d(\alpha(i),p) - d(\pi_{\alpha}(\alpha(i) a), p) \\
&\ge d(\alpha(i) a, \alpha(i))) - d(\alpha(i)a,q) - (A d(p, \alpha(i) a) + 2B) \\
&\ge |a| - (A+1) (2 \epsilon |a|+ K_{1}) - 2B > \epsilon |a|.
\end{aligned}
\]
This is a contradiction, so we must have $\tau \le t$. We then have \[
\begin{aligned}
d(\pi_{\alpha} (\alpha(i) a \beta(0)^{-1} \beta(j)), \alpha(i)) 
&\le d(\pi_{\alpha}(\gamma(M)),  \gamma(\tau)) + d(\gamma(\tau), \gamma(0)) \\
&\le \epsilon \log (|j| + |a|) + K_{1} + \frac{3}{2} |a| \\
&\le \epsilon \log |j| + 2|a|.
\end{aligned}
\]
\end{proof}

\section{Probabilistic part} \label{section:probablistic}

In this section, fixing a small enough $\epsilon>0$, we study the situation where a random path $(id =: Z_{0}, Z_{1}, \ldots, Z_{n})$ is seen by a superlinear-divergent direction, or to be precise, where $(Z_{i}, \ldots, Z_{i+\epsilon \log n})$ is (a part of) an $(f, \theta)$--divergent quasigeodesic and \[
(id, (Z_{i}, Z_{i+1}, \ldots, Z_{i+\epsilon \log n}), Z_{n})
\]
is $\epsilon^{2}$--aligned for some $i \ll n$. We will prove in Corollary \ref{cor:pivotEstimate} and Lemma \ref{lem:deviation} that this happens for an overwhelming probability.

To make an analogy, consider a random path $(id =: Z_{0}, Z_{1}, \ldots, Z_{n})$ arising from a simple random walk on the Cayley graph of a free group $F_{2} \simeq \langle a, b \rangle$. Here, we similarly expect that $Z_{n} = id$ is not desirable and $(id, (Z_{i}, Z_{i+1}), Z_{n})$ is aligned for some $i \ll n$. In fact, the alignment happens for all but exponentially decaying probability. A classical argument using martingales can be described as follows: \begin{enumerate}
\item construct a `score' that marks the progress made till step $i$;
\item prove that at each step $i$, it is more probable to earn a score rather than losing one.
\item sum up the difference at each step and use concentration inequalities to deduce an exponential bound.
\end{enumerate}
Here, the score at step $i$ should be determined by information up to time $i$. Moreover, when the score grows, the recorded local progresses should also pile up. To realize these features on a general Cayley graph other than tree-like ones, we employ the concatenation lemma proven in Section \ref{section:alignment}.

\subsection{Combinatorial model}\label{section:combinatorial}
In the sequel, let $\gamma$ be an $(f, \theta)$--divergent geodesic on $G$ with $\gamma(0) = id$ and $\epsilon >0$ be a small enough constant. Let us fix some constants: 
\begin{itemize}
    \item $K_{3} = K_{3}\left(f, \theta, q, Q, A, B, \epsilon^{3}, \epsilon \right)$ be as in Lemma \ref{lem:3segments};
    \item $K$ is larger than $K_{5} = K_{5}\left(\frac{1}{10q}\epsilon^{4} \right)$ and the twice of $K_{6} = K_{6}\left(0.1\epsilon^{4}\right)$ given by Lemma \ref{lem:linkage} and \ref{lem:1segment}, respectively;
    \item $N_{0}$ is a threshold such that \[
    \epsilon^{4} n > 10( K + K_{3}+ \log n)
    \]
    for all $n > N_{0}$.
\end{itemize} After multiple applications of our alignment lemmas, the power on $\epsilon$ will weaken, which is why we start with $\epsilon^4$.

Throughout this section, we will consider the following combinatorial model. Fix $w_{0}, w_{1}, \cdots \in G$. Now given a sequence of 3--tuples $s_{i} = (a_{i}, b_{i}, c_{i}) \in S^{3}$, we consider a word of the form \[
W_{k} = w_{0} \cdot a_{1}  \gamma(\epsilon \log n )  b_{1}  \gamma(\epsilon \log n)  c_{1} \cdot  w_{1} \cdots a_{k} \gamma(\epsilon \log n ) b_{k} \gamma(\epsilon \log n) c_{k}\cdot w_{k}.
\]
To ease the notation, let us also define \[
V_{k} = W_{k-1} a_{k}, \quad U_{k} = W_{k-1} a_{k} \gamma(\epsilon \log n) b_{k}
\]
We also denote 
\[s = (a_1, b_1, c_1, ..., a_k, b_k, c_k).\]

We will argue that for most choices of $s \in S^{3k}$, a certain subsequence of 
\[
\left(id,\, V_{1} \gamma|_{[0, \epsilon \log n]}, U_{1} \gamma|_{[0, \epsilon \log n]}\ldots, U_{k-1} \gamma|_{[0, \epsilon \log n]}, W_{k}\right)
\]
is well-aligned. In section \ref{section:randomwalks}, we will derive from this a deviation inequality (Lemma \ref{lem:deviation}), and deduce a central limit theorem. 

To show well-alignment, we argue analogously to \cite{boulanger2022large, gouezel2022exponential, choi2022random1}, by keeping track of times in which the random walk may travel along different translates of $\gamma|_{[0, \epsilon \log n]}$, and arguing that at most of these times, most directions of the random walk do not backtrack. To implement we need the following lemma \ref{prop:Schottky}. We remark that for the rest of the paper, whenever we discuss alignment of a sequence of points and geodesic segments, the only segments used are translates of $\gamma|_{[0, \epsilon \log n]}$. 

\begin{prop}\label{prop:Schottky}
Let $g \in G$ and let $n$ be an integer greater than $N_{0}$ and $|g|$. Let $S$ be the subset of $\mathcal{S}(id, K)$ described in Lemma \ref{lem:linkage} for $m=\epsilon \log n$. Then for any distinct $a,b \in S$, at least one of 
\[
\left(\gamma|_{[0, \epsilon \log n]}, \gamma(\epsilon \log n) ag\right) \quad\text{and}\quad \left(\gamma|_{[0, \epsilon \log n]}, \gamma(\epsilon \log n) bg \right) 
\]
is $\epsilon^{4} \log n$--aligned. Likewise, at least one of 
\[
\left(a^{-1}g, \gamma|_{[0, \epsilon \log n]}\right), \quad \text{and} \quad \left(b^{-1}g, \gamma|_{[0, \epsilon \log n]}\right)
\]
is $\epsilon^{4} \log n$--aligned.
\end{prop}

\begin{proof}
We prove the first claim only. Let $t \in \Z$ be such that $\gamma(t) = \pi_{\gamma}(\gamma(\epsilon \log n) a g)$. If $t$ is greater than $ \epsilon(1- \epsilon^{3} ) \log n$, we deduce that $\big(\gamma|_{[0, \epsilon \log n]},\, \gamma(\epsilon \log n) a g\big)$ is $\epsilon^{4} \log n$-aligned as desired. Let us deal with the remaining case: we assume \begin{equation}
    \label{eqn:propSchottky}
    t \in  (-\infty,\,\, \epsilon \log n - \epsilon^{4} \log n ].
\end{equation}
Consider two translates of $\gamma$: 
\[
\gamma_{1} =  a^{-1} \gamma(\epsilon \log n)^{-1} \gamma \qquad\text{and}\qquad \gamma_{2} = b^{-1} \gamma(\epsilon \log n)^{-1} \gamma,
\]
 and their subpaths 
\[
 \gamma_{1}' := \gamma_{1}|_{[t, \epsilon \log n]}
 \qquad \text{and} \qquad 
 \gamma_{2}' := \gamma_{2}|_{[0, \epsilon \log n]}.
 \] 
Let $\bar{\gamma}_{2}'$ be the reversal of $\gamma_{2}'$.

By the definition of $t$, $\left(\gamma(\epsilon\log n) a g, \gamma|_{[t, \epsilon \log n]}\right)$ is automatically $0$--aligned, or equivalently, $\left(g, \gamma_{1}'\right)$ is $0$--aligned. Next, since $a$ and $b$ are chosen from $S$, the subset of $\mathcal{S}(id, K)$ as described in Lemma \ref{lem:linkage}, we have that \[
	\pi_{\gamma}\left(\gamma(\epsilon \log n) a b^{-1}\right) \quad\text{is within}\quad  \mathcal{B}\left(\gamma(\epsilon \log n ), 0.1\epsilon^{4} |ab^{-1}|\right)
	\]
	and 
	\[
	\pi_{\gamma}\left(\gamma(\epsilon \log n) b a^{-1}\right) \quad\text{is within}\quad \mathcal{B}\left(\gamma(\epsilon \log n), 0.1\epsilon^{4} |ab^{-1}|\right).
	\] 
	Moreover, we have 
	\[
	|ba^{-1}| \ge 0.5K \ge K_{6}.
	\]
By plugging in $\alpha = \gamma$ and $\beta = \bar{\gamma}$ (i.e., $\beta(t) = \gamma(\epsilon \log n - t)$ for each $t \in \Z$), we can apply Lemma \ref{lem:1segment}. The required assumptions are 
\begin{align*}
 \pi_{\beta}\left(\beta(0) (ba^{-1})^{-1}\right)
 &= \pi_{\gamma}\left(\gamma(\epsilon \log n) a b^{-1}\right) \\
 &\in \mathcal{B}\left(\gamma(\epsilon \log n), 0.1\epsilon^{4} |ab^{-1}|\right) = \mathcal{B}\left(\beta(0), 0.1\epsilon^{4} \left|ab^{-1}\right|\right),
 \end{align*}
 and 
 \[
 \pi_{\alpha}\left(\alpha(\epsilon \log n), ba^{-1}\right) = \pi_{\gamma}\left(\gamma(\epsilon \log n), ba^{-1}\right) \in \mathcal{B}\left(\gamma(\epsilon \log n), 0.1 \epsilon^{4}\left|ab^{-1}\right|\right).
 \]
As a result, for each $j \in \Z$ we have
\[
d \Big(\pi_{\gamma}(\gamma(\epsilon \log n) ab^{-1} \gamma(\epsilon \log n)^{-1} \gamma(j)), \,\gamma(\epsilon \log n) \Big) \le 0.1\epsilon^{4} \log |j-\epsilon \log n| + 2|ab^{-1}|.
\]In other words, we have \[
\pi_{\gamma_{1}}(\gamma_{2}') \in \gamma_{1}\big( (\epsilon \log n - 0.1\epsilon^{4} \log (\epsilon \log n) - 2K, +\infty)\big).
\]Similarly we deduce that \[
\pi_{\gamma_{2}}(\gamma_{1}')\in \gamma_{2}\big( (\epsilon \log n -0.1\epsilon^{4} \log (\epsilon \log n) -2K, +\infty)\big).
\]We conclude that $(\gamma_{1}', \bar{\gamma}_{2}')$ is $(0.1\epsilon^{4} \log (\epsilon \log n) + 2K)$--aligned.

We now let $D = |g| + 2\epsilon \log n + 2K +K_{3}$; note that \[
D > \diam(g^{-1} \cup \gamma_{1}' \cup \gamma_{2}').
\]
Moreover, the lengths of $\gamma_{1}'$ and $\gamma_{2}'$ are at least $\epsilon^{4} \log n$ and we have \[
\epsilon^{4} \log n \ge \epsilon^{3} \log D.
\]
Finally, $(g, \gamma_{1}', \bar{\gamma}_{2}')$ is $(0.1\epsilon^{4} \log m + 2K)$--aligned, hence $0.2 \epsilon^{4} \log D$--aligned. Lemma \ref{lem:3segments} now tells us that $(g^{-1}, \bar{\gamma}_{2}')$ is $\epsilon^{4} \log D$--aligned. This implies that $(\gamma|_{[0, m]}, \gamma(m) b g)$ is $\epsilon^{4} \log n$--aligned as desired.
\end{proof}

Following Boulanger-Mathieu-Sert-Sisto \cite{boulanger2022large} and Gou{\"e}zel \cite{gouezel2022exponential}, we define the \emph{set of pivotal times} $P_{k}(s)$ inductively. We will suppress the notation $P_k := P_k(s)$ when it is unambiguous, and the remaining notation follows from the beginning of this section. First set $P_{0} = \emptyset$ and $z_{0} = id$. Given $P_{k-1} \subseteq \{1, \ldots, k-1\}$ and $z_{k-1} \in G$, $P_{k}$ and $z_{k}$ are determined by the following criteria. 
\begin{enumerate}[label=(\Alph*)]
\item When $\left(z_{k-1}, \,V_{k} \gamma|_{[0, \epsilon \log n]},\,U_{k} \gamma|_{[0, \epsilon \log n]}, W_{k}
\right)$ is $\epsilon^{3} \log n$--aligned, we set $P_{k} = P_{k-1} \cup \{k\}$ and $z_{k} = U_{k}$.
\item Otherwise, we find the maximal index $m \in P_{k-1}$ such that $(V_{m} \gamma|_{[0. \epsilon \log n]}, W_{k})$ is $\epsilon^{3} \log n$--aligned and let $P_{k} = P_{k-1} \cap \{1, \ldots, m-1\}$ (i.e., we gather all pivotal times in $P_{k-1}$ \emph{smaller than} $m$) and $z_{k} = V_{m}$. If such an $m$ does not exist, then we set $P_{k} = \emptyset$ and $z_{k} = id$.
\end{enumerate}

Given input $w_{0}, w_{1}, \ldots, w_{k} \in G$ and $s \in S^{3k}$, this algorithm outputs a subset $P_{k}(s)$ of $\{1, \ldots, k\}$. Our first lemma tells us that $P_{k}(s)$ effectively records the alignment.

\begin{lem}\label{lem:extremal}
The following holds for all $n > N_{0}$.

Let $P_{k} = \{i(1) < \ldots < i(M)\}$ and suppose that $\epsilon \log \left( |w_{0}| + \cdots + |w_{k}| + k\epsilon \log n \right) \le \log n$. Then there exist $g_{1}, \ldots, g_{N}= z_{k}$ such that $\left(V_{i(1)}, U_{i(1)}, \ldots, V_{i(M)}, U_{i(M)}\right)$ is a subsequence of $(g_{1}, \ldots, g_{N})$ and 
\[
(id,\, g_{1} \gamma|_{[0, \epsilon \log n]}, \ldots, g_{N} \gamma|_{[0, \epsilon \log n]}, W_{k})
\]
is $\epsilon^{2} \log n$--aligned.

\end{lem}

\begin{proof}
We induct on $k$. If we added a pivot, $P_{k} = P_{k-1} \cup \{k\}$, there are two cases: 
\begin{enumerate}
\item $P_{k-1}= \emptyset$. Then $(id, V_{k} \gamma|_{[0, \epsilon \log n]}, U_{k} \gamma|_{[0, \epsilon \log n]}, W_{k})$ is $(\epsilon^{3} \log n)$--aligned, with $z_{k} = U_{k}$, as desired.
\item $P_{k-1}= \{i(1) < \ldots < i(M-1)\}$ is nonempty. Then there exist $g_{1}, \ldots, g_{N}$ such that $\bigl(V_{i(1)}, \ldots, V_{i(M-1)}\bigr)$ is a subsequence of $(g_{1}, \ldots, g_{N})$, $g_{N} = z_{k-1}$ and 
\[
(id,\, g_{1} \gamma|_{[0, \epsilon \log n]}, \ldots, g_{N} \gamma|_{[0, \epsilon \log n]}, W_{k-1})
\]
is $\epsilon^{2} \log n$--aligned. Moreover, \[(z_{k-1}, V_{k}\gamma|_{[0, \epsilon \log n]}, U_{k} \gamma|_{[0, \epsilon \log n]}, W_{k})\] is $(\epsilon^{3} \log n)$--aligned. Here, since $(z_{k-1} \gamma|_{[0, \epsilon \log n]}, W_{k-1})$ is $\epsilon^{3} \log n$--aligned, 
\[
(z_{k-1} \gamma|_{[0, \epsilon \log n]}, W_{k-1} a_{k}) = (z_{k-1} \gamma|_{[0, \epsilon \log n]}, V_{k})
\]
is also $(\epsilon^{3} \log n + AK+B)$--aligned. Now Lemma \ref{lem:2segment} asserts that for large enough $n$, $(z_{k-1} \gamma|_{[0, \epsilon \log n]}, V_{k} \gamma|_{[0, \epsilon \log n]})$ is $\epsilon^{2}\log n$--aligned. As a result, \[
(id,\, g_{1} \gamma|_{[0, \epsilon \log n]}, \ldots, g_{N} \gamma|_{[0, \epsilon \log n]},\, V_{k}\gamma|_{[0, \epsilon \log n]}, \,U_{k}\gamma|_{[0, \epsilon \log n]}, W_{k})
\]
is $\epsilon^{2}\log n$--aligned, with $z_{k} = U_{k}$.
\end{enumerate}

Now suppose we backtracked: $P_{k} = P_{k-1} \cap \{1, \ldots, m-1\}$ for some $m \in P_{k-1}$. 
Letting $M = \#P_{k-1}$, so that $\#P_k = M+1$, our induction hypothesis tells us that there exist $g_{1}, \ldots, g_{N}$ such that $(V_{i(1)},U_{i(1)}, \ldots, V_{i(M+1)}, U_{i(M+1)})$ is a subsequence of $(g_{1}, \ldots, g_{N})$ and 
\[
(id,\, g_{1} \gamma|_{[0, \epsilon \log n]}, \ldots, g_{N} \gamma|_{[0, \epsilon \log n]}, W_{k-1})
\]
is $\epsilon^{2} \log n$--aligned. Moreover, we have that $(V_{i(M+1)} \gamma|_{[0, \epsilon \log n]}, W_{k})$ is $\epsilon^{3} \log n$--aligned by the criterion. It follows that \[
(id, \, g_{1} \gamma|_{[0, \epsilon \log n]}, \ldots, V_{i(M+1)} \gamma|_{[0, \epsilon \log n]}, W_{k})
\]
is $\epsilon^{2} \log n$--aligned, with $z_{k} = V_{m} = V_{i(M+1)}$, as desired.
\end{proof}

Next, we have \begin{lem}\label{lem:0thPivotCase}
Let us fix $a_{1}, b_{1}, c_{1}, \ldots, a_{k}, b_{k}, c_{k}$ and draw $a_{k+1}$, $b_{k+1}, c_{k+1}$ in $S^{3}$ according to the uniform measure. For $n \in \mathbb{N}$ sufficiently large, the probability that $\# P_{k+1} = \# P_{k} + 1$ is at least $9/10$.
\end{lem}
\begin{proof}
    We need to choose $a_{k+1}$, $b_{k+1}, c_{k+1}$ in $S^{3}$ such that 
    \[
    \left(z_{k}, \,V_{k+1} \gamma|_{[0, \epsilon \log n]},\,U_{k+1} \gamma|_{[0, \epsilon \log n]}, W_{k+1}\right)
    \] 
    is $\epsilon^{3} \log n$--aligned. By Proposition \ref{prop:Schottky}, there are at least 99 choices of $a_{k+1}$ such that 
    \[
    \left(z_{k}, \,V_{k+1} \gamma|_{[0, \epsilon \log n]}\right)
    \]
    is $\epsilon^{3} \log n$--aligned.

    Likewise, there are at least $98$ choices of $b_{k+1}$ such that both 
    \[
    (V_k, \,U_{k+1} \gamma|_{[0, \epsilon \log n]})
    \qquad\text{and}\qquad
    (V_{k+1} \gamma|_{[0, \epsilon \log n]},\,U_{k+1})
    \]
    are $\epsilon^4 \log n$--aligned. From lemma \ref{lem:2segment}, for sufficiently large $n$, this tells us there are at least $98$ choices of $b_{k+1}$ such that $(V_{k+1} \gamma|_{[0, \epsilon \log n]},\,U_{k+1} \gamma|_{[0, \epsilon \log n]})$ is $\epsilon^4 \log n$--aligned. 

    Finally, there are at least $99$ choices of $c_{k+1}$ such that $\left(U_{k+1} \gamma|_{[0, \epsilon \log n]}, W_{k+1}\right)$ is $\epsilon^3 \log n$--aligned.

    We are done as $\frac{99}{100}\cdot \frac{98}{100} \cdot \frac{99}{100} > \frac{9}{10}$.
\end{proof}

Given a sequence $s =( a_{i}, b_{i}, c_{i})_{i=1}^{k}$, we say that another sequence $s' = (a_{i}', b_{i}', c_{i}')_{i=1}^{k}$ is \emph{pivoted from} $s$ if they have the same pivotal times, $(a_{l}, c_{l}) = (a_{l}', c_{l}')$ for all $l=1, \ldots, k$, and $b_{l} = b_{l}'$ for all $l$ except for $l \in P_{k}(s)$. We observe that being pivoted is an equivalence relation.

\begin{lem}\label{lem:ChangingPivots}
Given $s =( a_{i}, b_{i}, c_{i})_{i=1}^{k}$ and a pivotal time $\ell \in P_k(s)$, construct a new sequence $s'$ by replacing $b_\ell$ with another $b'_{\ell}\in S$ such that 
\[\left(z_{\ell-1}, \,V_{\ell} \gamma|_{[0, \epsilon \log n]},\,U_{\ell} \gamma|_{[0, \epsilon \log n]}, W_{\ell}
\right)\]
is $\epsilon^{3} \log n$--aligned. Then $s'$ is pivoted from $s$. 
\end{lem}

\begin{proof}
    We need to show that both sequences $s$ and $s'$ have the same set of pivotal times. Before time $\ell$, the sequences are identical, so that $P_{j}(s) = P_{j}(s')$ for $j < \ell$. By our choice of $b'_\ell$, we know that the time $\ell$ is added as a pivot, and so $z'_\ell = U'_{\ell}$. Now we induct on $j > \ell$: suppose that all pivotal times in $P_{j-1}(s)$ are still in $P_{j-1}(s')$. 

    To determine $P_j(s)$, either we added a new pivotal time $j$ or we backtracked. In the former case, we have that $\left(z_{j-1}, \,V_{j} \gamma|_{[0, \epsilon \log n]},\,U_{j} \gamma|_{[0, \epsilon \log n]}, W_{j} \right)$ is $\epsilon^3 \log n$--aligned. Since $G$ acts on itself by isometries, this happens if and only if the sequence
    \[ \left(z'_{\ell}(z_{\ell}^{-1})z_{j-1}, z'_{\ell}(z_{\ell}^{-1})V_{j} \gamma|_{[0, \epsilon \log n]}, z'_{\ell}(z_{\ell}^{-1})U_{j} \gamma|_{[0, \epsilon \log n]}, z'_{\ell}(z_{\ell}^{-1})W_{j} \right) \] is $\epsilon^3 \log n$--aligned. But this is the same as requiring that 

     \[\left(z'_{j-1}, \,V'_{j} \gamma|_{[0, \epsilon \log n]},\,U'_{j} \gamma|_{[0, \epsilon \log n]}, W'_{j} \right)\] is $\epsilon^3 \log n$--aligned, so that $j \in P_j(s')$.

     In the latter case, we found the maximum $M$ such that $(V_{M} \gamma|_{[0. \epsilon \log n]}, W_{k})$ is $\epsilon^{3} \log n$--aligned. Since $\ell \in P_{k}(s) $, we know that $M > \ell$. Hence this is the same as requiring that  
     \[
     \left(z'_{\ell}(z_{\ell}^{-1})V_{M} \gamma|_{[0. \epsilon \log n]}, z'_{\ell}(z_{\ell}^{-1})W_{k}\right) 
     = \left(V'_{M} \gamma|_{[0. \epsilon \log n]}, W'_{k}\right)
     \]
     is $\epsilon^3 \log n$--aligned. Therefore $j \in P_{k}(s')$.
\end{proof}

Now fixing $w_{i}$'s, we regard $W_{k}$ as a random variable depending on the choice of 
\[
(a_{1}, b_{1}, c_{1}, \ldots, a_{k}, b_{k}, c_{k}),
\]
which are distributed according to the uniform measure on $S^{3k}$. 

Fixing a choice $s = (a_{1}, \ldots, c_k)$, let $\mathcal{E}_{k}(s)$ be the set of choices $s'$ that are pivoted from $s$. Since being pivoted is an equivalence relation, the collection of $\mathcal{E}_k(s)$'s partitions the space of sequences $S^{3k}$. We claim that most of these equivalence class are large: at pivotal times $\ell \in P_k$, one can replace $b_\ell$ with one of many other $b'_{\ell}$'s while remaining pivoted.

\begin{lem}\label{lem:expTailEst}
Let $s = (a_{1}, b_{1}, c_{1}, \ldots, a_{k}, b_{k}, c_{k})$. We condition on $\mathcal{E}_{k}(s)$ and we draw $(a_{k+1}, b_{k+1}, c_{k+1})$ according to the uniform measure on $S^{3}$. Then for all $j \ge 0$, \[
\Prob\left(\#P_{k+1}(s', s_{k+1}) < \#P_{k}(s') - j \, \big| \, (s', s_{k+1})\in \mathcal{E}_{k}(s) \times S^{3} \right) \le \left(\frac{1}{10} \right)^{j+1}.
\]
\end{lem}

We remark that the conditional measure $\Prob(\cdot| \mathcal{E}_k(s)\times S^3)$ on $S^{3(k+1)}$ is the same as the uniform measure on $\mathcal{E}_k(s) \times S^{3}\subset S^{3(k+1)}$, because $\Prob(\cdot)$ is the uniform measure on a finite set.

\begin{proof}
We induct on $j\geq 0$. The $j=0$ case is lemma \ref{lem:0thPivotCase}. We prove it for $j=1$. Suppose that we made some choice of $s_{k+1} := (a_{k+1}, b_{k+1}, c_{k+1})$ that lead to backtracking. We must show that for such an $s_{k+1}$,
\[
\Prob\left(\#P_{n+1}(s', s_{k+1}) < \#P_{n}(s') - 1 \, \big| \, s' \in \mathcal{E}_{k}(s)\right) \le \frac{1}{10}.
\]
To this end, we examine the final pivot $s_{\ell}$. By Lemma \ref{lem:ChangingPivots}, we can replace $b_{\ell}$ with any distinct $b'_{\ell} \in S$ such that \[\left(z_{\ell-1}, \,V_{\ell} \gamma|_{[0, \epsilon \log n]},\,U_{\ell} \gamma|_{[0, \epsilon \log n]}, W_{\ell}\right)\]
is $\epsilon^{4} \log n$--aligned. There are at least $98$ choices of such a $b'_{\ell}$, by Proposition \ref{prop:Schottky}.

Likewise, there are at least $98$ choices of $b'_{\ell} \neq b_{\ell}$ such that $(U_{\ell}\gamma_{[0, \epsilon \log n]}, W_k)$ is $\epsilon^4 \log n$--aligned. From Lemma \ref{lem:3segments}, we know that 

\[\left(V_{\ell} \gamma|_{[0, \epsilon \log n] }, \,W_k\right)\] 
is $\epsilon^3 \log n$--aligned. For this choice of $s'$, we have $P_{k+1}(s') = P_{k}(s') \cap \{0, ..., \ell - 1\}$. In particular, $\#P_{k+1} = \#P_{k} - 1$.
Hence 
\[
\Prob\left(\#P_{k+1} < \#P_{k} - 1 \, \big| \, \mathcal{E}_{k}(s), s_{k+1}\right) \le \left(\frac{4}{100}\right) < \left(\frac{1}{10}\right).
\]
To handle the induction step for $j \geq 2$, the same argument works, except we condition not only on $s_{k+1}$ but also on the final $j$ pivotal increments which resulted in backtracking.
\end{proof}

\begin{cor}\label{cor:pivotEstimate}
\[
\Prob \left( \#P_{k} \le k/2 \right) < (1/10)^{k}.
\]
\end{cor}

\subsection{Random walks}\label{section:randomwalks}

Recall that $G$ contains an $(f, \theta)$--superdivergent $(q,Q)$--quasigeodesic $\gamma : \Z\rightarrow G$ with $\gamma(0) = id$. 

Let $\mu$ be a probability measure on $G$ whose support generates $G$ as a semigroup. Passing to a convolution power if necessary, assume that $\mu(a) > 0$ for all $a$ in our finite generating set $A \subset G$. Let $(Z_n)_{n\geq1}$ be the simple random walk generated by $\mu$, and let $\alpha \in (0,1)$. We can define \[\begin{aligned}
p &= \min \{\mu(a), a \in A\},\\
\epsilon &=\frac{\alpha/100}{\log(1/p)}.
\end{aligned}
\]
so that $p^{\epsilon \log n} = n^{-\alpha/100}$. Then for any path $\eta$ of length $100 \epsilon \log n$ and any $k \in \Z$, we have \[
\Prob( (g_{k+1}, \ldots, g_{k+100 \epsilon \log n}) = \eta ) \ge n^{-\alpha}.
\]
Also recall that for any three points $o,x,y \in G$ we can define the Gromov product, given by \[(x, y)_{o} = \frac{1}{2}\left(d(o,x) + d(o,y) - d(x,y) \right).\]

We now have:

\begin{lem}\label{lem:deviation}
For any $0<\alpha < 1$, there exist $K > 0$ such that for each $x \in \mathcal{B}(id, 2n)$ we have \[
\Prob\left[(x, Z_{n})_{id} \ge n^{3\alpha})\right] \le K e^{-n^{\alpha}/K}.
\]
\end{lem}

\begin{proof}
First, we would like to find a nice decomposition of our random walk, which will allow us to analyze the sample paths using our combinatorial model in section \ref{section:combinatorial}.

Let $\overline{\lambda}_{i}$ be i.i.d. distributed according to the uniform measure on the subset $S' \subset G^5$ defined by \[
S' := \{(a, \gamma', b, \gamma', c) : a, b, c \in S, \gamma' = \gamma(\epsilon \log n)\}.
\]

Then the evaluation $\lambda_i = a\cdot\gamma '\cdot b \cdot \gamma' \cdot c$ is distributed according to the measure $\mu_S * \gamma' * \mu_S * \gamma' * \mu_S$, where $\mu_S$ is uniform over $S$.

Let $N = 3K + 2\epsilon \log n$. By our choice of $p$, for each $a,b,c \in S$ we have $\mu^{*N}(a\gamma ' b \gamma' c) \geq p^{N}$. Then we can decompose 
\[\mu^{*N} = 10^6p^{N} \left(\mu_S * \gamma' * \mu_S * \gamma' * \mu_S\right) + (1-10^6 p^{N}) \nu,\]

for some probability measure $\nu$.

Now we consider the following coin-toss model, 
Let $\rho_{i}$ be independent 0-1 valued random variables, each with probability $10^{6} \cdot p^{N}$ of being equal to 1. Also let $\xi_{i}$ be i.i.d. distributed according to $\nu$.  
We set 
\[g_i = \left\{\begin{array}{cc} \lambda_{i} & \textrm{if $\rho_{i} = 1$} \\
\xi_{i} & \textrm{otherwise}. \end{array}\right.\]

Then $(g_1 \cdots g_n)_n$ has the same distribution as $(Z_{Nn})$, because each $g_i$ is distributed according to $\mu^{*N}$. 

Hoeffding's inequality tells us that \[
\Prob\left( \sum_{i=1}^{n^{3\alpha}} \rho_{i} \ge 0.5 n^{3\alpha} \cdot n^{-\alpha}\right) \ge 1 - 2 \textrm{exp} \left( - \frac{2 (0.5 n^{2\alpha})^{2}}{n^{3\alpha}}\right) \ge 1 - 2 \exp (-0.5 n^{\alpha}).
\]
After tossing away an event of probability at most $2\exp(-0.5n^{\alpha})$, we assume $\sum_{i=1}^{n^{3\alpha}} \rho_{i} \ge 0.5 n^{2\alpha}$. 

To apply the analysis of our combinatorial model, we condition on the values of $\rho_i, \xi_{i}$ and only keep the randomness coming from the $\eta_{i}$'s. Let \[i(1) < i(2) < \ldots < i(M)\] be the indices in $[1, n^{3\alpha}]$ where $\rho_i = 1$. Then we can write \[
x^{-1} \cdot Z_{n} =  w_{0} \cdot a_{1}  \gamma(\epsilon \log n )  b_{1}  \gamma(\epsilon \log n)  c_{1} \cdot  w_{1} \cdots a_{M} \gamma(\epsilon \log n ) b_{M} \gamma(\epsilon \log n) c_{M}\cdot w_{M},
\]
where \[
\begin{aligned}
w_{0} &= x^{-1}  g_{1} \cdots g_{N(i(1)-1) - 1},\\
w_{1} &= g_{Ni(1) + 1} \cdots g_{N(i(2)-1) - 1}, \\
&\vdots \\
w_{M} &= g_{Ni(M) + 1} \cdots g_{n}
\end{aligned}
\]
and $a_{i}, b_{i}, c_{i}$ are i.i.d.s distributed according to the uniform measure on $S$. As in the previous section, we set $s = (a_1,b_1,c_1, \ldots, a_M,b_M, c_M)$. By Lemma \ref{cor:pivotEstimate}, the set of pivots $P_{M}(s)$ is nonempty with probability at least $1-(1/10)^{M} \ge 1- (1/10)^{0.5 n^{2\alpha}}$. By Lemma \ref{prop:induction}, for any pivotal time $i \in P_M(s)$ we have 
\[ 
\left(id, x^{-1} Z_{N(i-1)} \gamma|_{\epsilon \log n}, x^{-1} Z_{n}\right) 
   = \left(id, (x^{-1}Z_{N(i-1)}, \ldots, x^{-1}Z_{N(i-1)+\epsilon \log n}), x^{-1} Z_{n}\right)\] 
is $\epsilon \log n$--aligned. Lemma \ref{lem:quickLanding2} implies that $[id, x^{-1} Z_{n}]$ passes 
through the $K_{1}$--neighborhood of $(x^{-1}Z_{N(i-1)}, \ldots, x^{-1}Z_{N})$.  In other words, $[x, Z_{n}]$ passes through the $(Ni  + K_{0})$--neighborhood of $id$, which is within the $n^{3\alpha}$--neighborhood of $id$ when $n$ is large.
\end{proof}

\begin{cor}\label{cor:deviation}
For any $\alpha>0$, there exists $K'$ such that for each $0 \le m \le n$ we have
\[
\E\left[(id, Z_{n})_{Z_{m}}^{2}\right] \le n^{6\alpha} + K e^{-n^{\alpha}/K} \cdot n \le n^{6\alpha}+K'.
\]
\end{cor}

The following lemma states that our deviation inequality (Corollary \ref{cor:deviation}) implies a rate of convergence in the subadditive ergodic theorem.

\begin{lem}\label{lem:cute}
Let \[L := \lim_{n\to \infty}\frac{1}{n} \E[d(id, Z_{n})].\] Then \[L - \frac{1}{n}\E[d(id, Z_{n})]= o\left(\frac{1}{\sqrt{n}}\right).\]
\end{lem}

\begin{proof}
Note that by the definition of the Gromov product, we have \[
\E[d(id, Z_{n2^{k}})] 
= \sum_{i=1}^{2^{k}} \E\left[d(Z_{n(i-1)}, Z_{ni})\right] 
- 2\sum_{i=1}^{k} \sum_{t = 1}^{2^{k-i}} \E\left[(Z_{n 2^{i}(t-1))}, Z_{n2^{i}t})_{Z_{n(2^{i}t - 1)}}\right].
\]
Also by corollary \ref{cor:deviation}
\[
\E\left[(Z_{n 2^{i}(t-1))}, Z_{n2^{i}t})_{Z_{n(2^{i}t - 1)}}\right] \le 2(n2^{i-1})^{6\alpha} + K'
\] 
and we also know that $\E[d(Z_{n(i-1)}, Z_{ni})] = \E[d(id, Z_{n})]$ for any $i \in \mathbb{N}$. Hence for any sufficiently small $\alpha>0$, we have \[\begin{aligned}
\left|\frac{1}{2^{k}} \E[d(id, Z_{n2^{k}})] - \E[d(id, Z_{n})]\right| &\le \frac{2}{2^{k}} \sum_{i=1}^{k} 2^{k-i} \cdot (2(n2^{i-1})^{6\alpha} + K') \\
&\lesssim n^{6\alpha} \sum_{i=1}^{k} 2^{-i/2}.
\end{aligned}
\] 
As $k \rightarrow \infty$ the quantity $2^{-k} \E[d(id, Z_{n2^{k}})]$ converges to $L$. Picking $\alpha < 1/12$, we can send $k \rightarrow \infty$ and divide by $n$ to conclude.
\end{proof}

We now prove the CLT (Theorem \ref{thm:main}). It is essentially the same argument as \cite{Mathieu2020deviation}, but with a different deviation inequality as input.

\begin{proof}
We claim that for any $\epsilon>0$, there exists $N$ sufficiently large, such that the sequence 
\[
\frac{1}{\sqrt{Nk}} \bigl(d(id, Z_{Nk}) - \E[d(id, Z_{Nk})]\bigr)
\]
converges to a Gaussian distribution up to an error at most $\epsilon$ in the L{\'e}vy distance.

Indeed, the sequence 
\[
\left\{\frac{1}{\sqrt{k}} \big(d(id, Z_{k}) - \E[d(id, Z_{k})]\big)\right\}_{k > 0}
\]
is eventually $\epsilon$--close to a distribution $X$ (in the L{\'e}vy distance) if and only if its $N$--jump subsequence $\big\{\frac{1}{\sqrt{Nk}} \big(d(id, Z_{Nk}) - \E[d(id, Z_{Nk})] \big)\big\}_{k > 0}$ is as well. Moreover, from Lemma \ref{lem:cute}, we know that \[\E[d(id, Z_{Nk}) = L Nk + o(1/\sqrt{Nk}).\] 

To show the claim, we first take a sequence \[
0 = i(0) < i(1) < \ldots < i(2^{\lfloor \log_{2} k \rfloor }) =k
\]
such that $i(t+1) - i(t) = 1$ or $2$ for each $t$. The easiest way is to keep halving the numbers, i.e., 
\[
i(2^{t}k) := \left\lfloor \frac{i(2^{t}(k-1)) + i(2^{t}(k+1))}{2} \right\rfloor
\]
for each $t$ and odd $k$. Let $T$ be the collection of $i(t)$'s such that $i(t+1) - i(t) = 2$.

Then, \[
\frac{1}{\sqrt{Nk}} \big(d(id, Z_{Nk}) - \E[d(id, Z_{nk})] \big) = I_1 - I_2 - I_3
\]

where 

\[I_1 = \sum_{i=1}^{k} \frac{1}{\sqrt{k}} \left[ \frac{d(Z_{N(i-1)}, Z_{Ni}) - \E[d(Z_{N(i-1)}, Z_{Ni})]}{\sqrt{N}}\right]\]

\[I_2 =  \frac{2}{\sqrt{Nk}}\sum_{t \in T} \left( (Z_{Ni(t)}, Z_{N(i(t)+2})_{Z_{N(i(t)+1})} -  \E[(Z_{Ni(t)}, Z_{N(i(t)+2})_{Z_{N(i(t)+1})}] \right), \]

and 

\[I_3 = \frac{2}{\sqrt{Nk}}\sum_{t=1}^{\lfloor \log_{2} k - 1 \rfloor} \sum_{l = 1}^{2^{\lfloor \log_{2} k \rfloor - t - 1}} \!\! \Big((Z_{N2^{t}l}, Z_{N2^{t} (l+2)})_{Z_{N2^{t}(l+1)}} - \E\left[(Z_{N2^{t}l}, Z_{N2^{t} (l+2)})_{Z_{N2^{t}(l+1)}}\right] \Big).\]

We claim that for sufficiently large $N \in \mathbb{N}$, $I_2$ and $I_3$ are small (in terms of the L{\'e}vy distance). Then the only non-negligible term $I_1$ is a sum of i.i.d random variables, normalized to converge to a Gaussian as $k \to \infty$.

The second summation $I_{2}$ is the sum of at most $k$ independent RVs whose variance is bounded by 
\[\frac{4}{Nk} \cdot 3N^{6\alpha}. \] Hence, the second summation has variance at most 
$12 N^{6\alpha-1}$ and 
\[\Prob(|I_{2}| \ge N^{-\alpha}) \le 12 N^{8\alpha - 1}\] by Chebyshev.

Now for each $t$, \[
I_{3;t} := \frac{2}{\sqrt{Nk}}\sum_{l = 1}^{2^{\lfloor \log_{2} k \rfloor - t - 1}} \Big((Z_{N2^{t}l}, Z_{N2^{t} (l+2)})_{Z_{N2^{t}(l+1)}} - \E[(Z_{N2^{t}l}, Z_{N2^{t} (l+2)})_{Z_{N2^{t}(l+1)}}] \Big)
\]is the sum of at most $k/2^{t}$ independent RVs whose variance is bounded by $\frac{4}{Nk} \cdot 3 (N2^{t})^{6\alpha}$. This means that $I_{3;t}$ has variance at most $12 N^{6\alpha - 1} \cdot 2^{(6\alpha - 1) t}$, and 
\[\Prob(|I_{3; t}| \ge N^{-\alpha} 2^{-\alpha t}) \le 12 N^{8 \alpha - 1} 2^{(8\alpha - 1) t}\]
by Chebyshev.

Summing them up, we have 
\[|I_{2} + \sum_{t} I_{3;t}| \le N^{-\alpha} \sum_{t} 2^{-\alpha t}\]
outside a set of probability $N^{8\alpha - 1} \sum_{t} 2^{(8 \alpha - 1) t}$. These are small, regardless of the range of $t$. More precisely, by setting $\alpha = 1/10$, we deduce that 
\[|I_{2} + I_{3}| = O(N^{-1/10})\]
outside a set of probability $O(N^{-1/10})$, ending the proof.
\end{proof}

We now prove the CLT for random walks with finite $p$-th moment for some $p>2$. It suffices to show that Corollary \ref{cor:deviation} holds for such random walks.

For some $q > 0$, let $E$ be the event that $\sum_{i=1}^{n} |g_{i}|$ is at least $n^{q}$. We note the following inequality \[\begin{aligned}
\E\left[ n^{q (p-2)} \left( \sum_{i=1}^{n} |g_{i}| \right)^{2} 1_{\sum_{i=1}^{n} |g_{i}| \ge n^{q}} \right] 
&\le \E \left[ \left( \sum_{i=1}^{n} |g_{i}| \right)^{p} \right] \\
&\le \E \left[ \left( n \max_{i=1}^{n} |g_{i}|\right)^{p}\right]\\
&\le n^{p} \E \left[ \sum_{i=1}^{n} |g_{i}|^{p}\right]\\
&\le n^{p+1} \E |g|^{p}.
\end{aligned}
\]
This implies that 
\[
\E\left[ \left(\sum_{i=1}^{n} |g_{i}| \right)^{2} 1_{E}\right] \le Cn^{(p+1) - q(p-2)}.
\]
By taking $q > \frac{p+1}{p+2}$, we can keep this bounded.

Now on the event $E^c$, we argue as in Lemma \ref{lem:deviation}. We remark that the only place we used the finite support assumption was to invoke Lemma \ref{lem:extremal}. In particular, we needed 
\[\epsilon \log \left( |w_{0}| + \cdots + |w_{k}| + k\epsilon \log n \right) \le \log n,\] where $w_i$. However, on the event $E^c$, this assumption is still met, replacing $\epsilon$ with $\epsilon/q$ if necessary. Then we may still apply lemma \ref{lem:extremal}. Hence, we get \[
\E[(id, Z_{n})_{Z_{m}}^{2}] \le n^{6\alpha} + K e^{-n\alpha/K} \le 2n^{6\alpha} + K'.
\]
Given this estimate, we get:

\begin{thm}
Let $\mu$ be an admissible measure on $G$ with finite $p$--moment for some $p > 2$, and $(Z_{n})_{n}$ be the random walk on $G$ generated by $\mu$. Then there exist constants $\lambda, \sigma$ such that \[
\frac{d_{X}(o, Z_{n} o)-Ln}{\sigma \sqrt{n}} \rightarrow \mathcal{N}(0, 1).
\]
\end{thm}

\appendix
\section{Right-angled Coxeter Groups} \label{appendix:racg}

Let $\Gamma = (V,E)$ be a finite simple graph. We can define the \emph{Right-angled Coxeter group} by the presentation
\[
W_{\Gamma} = \langle v \in V|v^2, [v,w], (v,w) \in E\rangle.
\]

In this appendix we show the following

\begin{lem}\label{lem:RACG}
    Let $W_{\Gamma}$ be a Right-angled Coxeter group of thickness $ k \geq 2$. Then any Cayley graph of $\Gamma$ contains a periodic geodesic $\sigma$ which is $(f, \theta)$--divergent for some $\theta>0$ and $f(r) \gtrsim r^k$.
\end{lem}

We only need to slightly modify the proof of Theorem C given in \cite{Levcovitz2022thick}. They show that a RACG of thickness at least $k$ has divergence at least polynomial of degree $k+1$. To do this, they construct a periodic geodesic $\gamma$ such that for any path $\kappa$ with endpoints on $\gamma$ and avoiding an $r$-neighbourhood of $\gamma$'s midpoint, any segment of $\kappa$ with projection at least some constant has to have length at least $r^{k}$. By integrating they get $r^{k+1}$. For completeness, we include the proof below.

\begin{proof}
Since the claim is quasi-isometry invariant, we work on the Davis complex $\Sigma_\Gamma$. We modify the proof of Theorem C of \cite{Levcovitz2022thick}, borrowing their notation and terminology. Take the word $w$ in the proof, so that $\sigma$ is a bi-infinite geodesic which is the axis of $w$, and set $p_i = w^i$. Since the Davis complex is a CAT(0) cube complex, the nearest point projection $\pi: \Sigma_{\Gamma} \to \sigma$ is well-defined and $1$--Lipschitz. 

Let $\kappa: [0, t] \to \Sigma_\Gamma$ be a path whose projection has diameter at least $2|w|$, which is disjoint from the $|w|r$-neighbourhood around some $w^i$. As the projection of $\kappa$ has length at least $2|w|$, we can find some points $p_j, p_k$ such that
\[
\pi(\kappa(0)) < p_j < p_k < \pi(\kappa(t))
\] 
in the orientation on $\sigma$. Here $p_j, p_k = w^j, w^k$. 

For the rest of the proof, we follow \cite{Levcovitz2022thick}. For some $j \leq i < k$, let $H_i$ (resp. $K_i$) be the hyperplane dual to the edge of $\sigma$ which is adjacent to $p_i$ (resp. $p_{i+1}$) and is labeled by $s_0$ (resp. $s_n$). As hyperplanes separate $\Sigma_\Gamma$ and do not intersect geodesics twice, it follows that $H_i$ (resp. $K_i$) intersects $\kappa$. Let $e_i$ (resp. $f_i$) be the last (resp. first) edge of $\kappa$ dual to $H_i$ (resp. $K_i$). Let $\gamma_i$ (resp. $\eta_i$) be a minimal length geodesic in the carrier $N(H_i)$ (resp. $N(K_i)$) with starting point $p_i$ (resp. $p_{i+1}$) and endpoint on $e_i$ (resp. $f_i$). Let $\alpha_i$ be the subpath of $\kappa$ between $\gamma_i \cap \kappa$ and $\eta_i \cap \kappa$. As $w$ is a $\Gamma$--complete word, no pair of hyperplanes dual to $\sigma$ intersect. By our choices, $\alpha_i \cap \alpha_j$ is either empty or a single vertex for all $i \neq j$. Let $D_i$ be the disk diagram with boundary path $\gamma_i \alpha_i \eta_i^{-1} \beta_i$ where $\beta_i$ has label $w^{-1}$. For each $0 \leq i \leq r-2$, we observe the following:
\begin{enumerate}
	\item The path $\gamma_i$ is reduced. 
	\item By Lemma 7.2, no $(k-1)$--fence connects $\gamma_i$ to $\eta_i^{-1}$ in any disk diagram with boundary path $\gamma_i \alpha_i \eta_i^{-1} \beta_i$. 
	\item The path $\alpha_i$ does not intersect the ball $B_{p_i}(|w|(r))$. 
\end{enumerate}
Thus we can apply \cite[Theorem 6.2]{Levcovitz2022thick} to $D_i$ by setting, in that theorem, 
\[
\gamma = \gamma_i,\quad \alpha = \alpha_i,\quad \eta = \eta_i^{-1},\quad \beta = \beta_i, \quad\text{and}\quad  L = k - 1 R = |w|(r - i).
\]
We conclude that for $r$ large enough
\[
|\alpha_i| \geq C' (|w|(r)^k).
\]
 As $\alpha_i$ is a subsegment of $p$, we are done. 
\end{proof}

\section{Superlinear-divergence and strongly contracting axis} \label{appendix:contracting}

In this section, we give two constructions that illustrates the logical independence between superlinear divergence and strongly contracting property. We first recall the notion of strongly contracting geodesics.

\begin{definition}[Strongly contracting sets]\label{dfn:contracting}
For a subset $A \subseteq X$ of a metric space $X$ and $\epsilon > 0$, we define the \emph{closest point projection} of $x \in X$ to $A$ by \[
\pi_{A}(x) := \big\{a \in A : d_{X}(x, a)= d_{X}(x, A) \big\}.
\]
$A$ is said to be \emph{$K$-strongly contracting} if: \begin{enumerate}
\item $\pi_{A}(z) \neq \emptyset$ for all $z\in X$ and
\item for any geodesic $\eta$ such that $\eta \cap N_{K}(A) = \emptyset$, we have $\diam(\pi_{A}(\eta) ) \le K$.
\end{enumerate}
\end{definition}

\begin{lem}\label{lem:gerstenSuper}
There exists a finitely generated group $G$ containing an element  whose axis is strongly contracting but not superlinear-divergent.
\end{lem}

\begin{proof}
Let $G$ be the group constructed by Gersten in \cite{gersten1994quadratic}: \[
G = \langle x, y, t \,|\, txt^{-1} = y, xy=yx \rangle.
\]
The group $ G $ naturally acts on the universal cover of its presentation complex, which is a CAT(0) cube complex. Recall that the presentation complex of $ G $ is defined as follows: start with a single $0 $-cell, attach a $ 1 $-cell for each of the three generators $ x,y,t $, and attach a $ 2 $--cell for each of the relations $ [x,y]$ and $txt ^{-1}y ^{-1}  $. Let $ X $ be the universal cover of this complex, which Gersten shows is CAT(0) \cite[Prop. 3.1]{gersten1994quadratic}.
	 
The induced combinatorial metric on $ X $ is isometric to the word metric with respect to $\{x, y, t\}$.

Let $g = tx$ and $\gamma$ be a path connecting $(\ldots, id, t, tx, txt, (tx)^{2}, \ldots)$. Then $\gamma$ is a $g$--invariant geodesic, and $\gamma$ does not bound a flat half-plane (the cone angle of $\gamma$ at its each vertex is $3\pi/2$). Hence, $\gamma$ is rank-1 and we can conclude that $g$ is strongly contracting.

Meanwhile, by \cite[Theorem 4.3]{gersten1994quadratic}, $G$ has quadratic divergence. Given an appropriate action of $ G $ on a hyperbolic space, we would be able to conclude from \cite[Lemma 3.6]{goldsborough2021markov} that $ \gamma $ is not superlinear-divergent. Since we do not assume a hyperbolic action, we instead present a modification of Goldborough-Sisto's argument.

Suppose that there exists an $(A, B)$--coarsely Lipschitz projection $\pi_{\gamma}: G \rightarrow \gamma$, a constant $\theta > 0$ and a superlinear function $f$ such that $\gamma$ is $(f, \theta)$--divergent with respect to $\pi_{\gamma}$. Up to a finite additive error, we may assume that $\pi_{\gamma}$ takes the values $\{(zx)^{i} : i \in \Z\}$.

Let $\epsilon = \frac{1}{2(A+3)}$ and let $n$ be a sufficiently large integer. We claim: 

\begin{claim}\label{claim:gerstenClaim}
    If a point $p \in G \setminus B(id, n)$ satisfies $d(p, \gamma) \le \epsilon n$, then $\pi_{\gamma}(p) = (tx)^{i}$ for some $|i| > 0.5n$. 
\end{claim}

\begin{proof}[Proof of Claim \ref{claim:gerstenClaim}]First, from $d(p, \gamma) \le \epsilon n$ and the coarse Lipschitzness of $\pi_{\gamma}$, we deduce \[
d(p, \pi_{\gamma}(p)) \le (A+1) \epsilon n + 2B.
\]
Hence, we have \[
d(id, \pi_{\gamma}(p)) \ge d(id, p) - d(p, \pi_{\gamma}(n)) \ge n - (A+1) \epsilon n - 2B > 0.5n
\]
and the claim follows.
\end{proof}

Next, we let  \[
	a_{n} = (tx)^{(1-\epsilon)n} y^{-\lfloor \epsilon n \rfloor}, b_{n} = (tx)^{-(1-\epsilon)n} y^{-\lfloor \epsilon n \rfloor}
\]
and let $\eta$ be an arbitrary path in $G \setminus B(id, n)$ connecting $a_{n}$ and $b_{n}$. Let $m, m' \in \Z$ be such that $\pi_{\gamma}(a_{n}) = (tx)^{m}$ and $\pi_{\gamma}(b_{n}) = (tx)^{m'}$. We then have \[
d((tx)^{n}, \pi_{\gamma}(a_{n})) \le d((tx)^{n}, a_{n}) + d(a_{n}, \pi_{\gamma}(a_{n})) \le  (A+2) \epsilon n + 2B < 0.5 n.
\]
It follows that $m > n - 0.5n \ge 0.5n$. Similarly, we deduce $m' < -0.5n$.

We examine the two connected components of $\eta \setminus N_{\epsilon n}(\gamma)$ as well as $\eta \cap N_{\epsilon n}(\gamma)$. Each component of $\eta \cap N_{\epsilon n}(\gamma)$ attains values of $\pi_{\gamma}(\cdot)$ in 
\[
\{(tx)^{i} : i < -0.5 n\} \qquad \text{or} \qquad\{(tx)^{i} : i > 0.5n\},
\] 
by Claim \ref{claim:gerstenClaim}, but not in both (by the coarse Lipschitzness of $\pi_{\gamma}$). Meanwhile, the endpoints of $\eta$ attain values of $\pi_{\gamma}(\cdot)$ in $\{(tx)^{i} : i < -0.5 n\}$ and $\{(tx)^{i} : i > 0.5n\}$, respectively. As a result, there exists a subsegment $\eta'$ of $\eta$, as a component of $\eta \setminus N_{\epsilon n}(\gamma)$, such that 
\[
\pi_{\gamma}(\eta'^{+}) \in \{(tx)^{i} : i > 0.5n\}
\qquad \text{and} \qquad 
\pi_{\gamma}(\eta'^{-}) \in \{(tx)^{i} : i < -0.5n\}.\]
It follows that the length of $\eta'$ is at least $(n/\theta) \cdot f(\epsilon n)$. Since $\eta$ is longer than $\eta'$, we deduce that an arbitrary path in $G \setminus B(id, n)$ connecting $a_{n}, b_{n} \in B(id, n)$ is longer than $(n / \theta) \cdot f(\epsilon n)$. When $n$ increases, this contradicts the quadratic divergence of $G$. Hence, we deduce that $\gamma$ is not superlinear-divergent.
\end{proof}

\begin{lem}
There exists a proper geodesic metric space $X$ that contains a superlinear-divergent geodesic $\gamma$ that is not strongly contracting.
\end{lem}

\begin{proof}
Let $X = \mathbb{H}^{2}$ and $\gamma$ be a bi-infinite geodesic $\gamma$ on $X$ with respect to the standard Poincar{\'e} metric $ds_{0}^{2}$. Let $o \in \gamma$ be a reference point on $\gamma$ and  let $\proj_{\gamma}$ be the closest point projection onto $\gamma$ with respect to $ds_{0}^{2}$. For each $x \in X$, let $r$ be the (directed) distance from $x$ to $ \gamma $ and let $\tau$ be the (directed) distance from $o$ to $\proj_{\gamma}(x)$. Since $(r, \tau)$ is an orthogonal parametrization of $X$, there exists a continuous coefficient $F_{0}$ such that \[
ds_{0}^{2} = dr^{2} + F_{0}(x) d\tau^{2}
\]
holds at each point $x \in X$. We note that $F_{0}(x) \sim e^{\kappa r(x)}$ for some $\kappa > 0$ (due to the Gromov hyperbolicity of $(X, ds_{0}^{2})$) and $F_0(x) \ge 1$.

We will now define a new metric $ds^{2}$ as follows. For each $i > 0$ and $j \in \Z$ let 
\[
I_{i, j} = \{(r, \tau) : r = 4^{2^{i}}, 2j +i\le \tau \le  2j + i+1\},
\]
and let 
\[
S := \bigcup_{i > 0, \ j \in \Z} I_{i, j}.
\]
Let $\chi: \mathbb{R}^{2} \rightarrow [0, 1]$ be a smooth function that takes value 0 on $S$ and 1 on $\mathbb{R}^{2} \setminus N_{0.1}(S)$. We finally define 
\[
F(x) := F_{0}(x) \cdot \chi(r(x), \tau(x)) + (1 - \chi(r(x), \tau(x)))
\]
and 
\[
ds^{2} := dr^{2} + F(x) d\tau^{2}.
\]
First, $\proj_{\gamma}$ is still the closest point projection with respect to $ds^{2}$. Indeed, the shortest path from $x \in X$ to $\gamma$ is the one that does not change in the value of $\tau$. As a corollary, the $K$-neighborhoods of $\gamma$ with respect to the two metrics coincide.

Let $i$ be a positive integer and let $x, y \in X$ be such that $r(x) = r(x) = 4^{2^{4i}}$ and $\tau(x) = 0$, $\tau(y) = 2i$. We first consider a path $\eta$ connecting $x$ to $y$ while passing through $N_{K}(\gamma)$. Then the total length is at least $2 \cdot (4^{2^{4i}} -K)$. Next, we take a piecewise geodesic path $\eta'$ that goes like: 
\begin{align*}
(r, \tau) = (4^{2^{4i}}, 0) &- (4^{2^{4i}}, 1) - (4^{2^{4i-1}}, 1) - (4^{2^{4i-1}}, 2) - \cdots \\
 &- (4^{2^{3i+1}}, i) - (4^{2^{3i}}, i) - (4^{2^{3i}}, i+1) - (4^{2^{3i+1}}, i+1) - \cdots - (4^{2^{i}}, 2i).
\end{align*}
Then the total length is $2(4^{2^{4i}} - 4^{2^{3i}}) + 2i$. Note also that $\eta'$ does not intersect $N_{K}(\gamma)$. We conclude that the geodesic connecting $x$ to $y$ does not touch $N_{K}(\gamma)$. Note also that the projection is larger than $2i$. By increasing $i$, we conclude that $\gamma$ is not $K$--strongly contracting for any $ K>0 $.

Meanwhile, it is superlinear-divergent. To see this, suppose a path $\eta$ lies in $X \setminus N_{R}(\gamma)$ and satisfies $\pi_{\gamma}(\eta) > 4$. Then $\pi_{\gamma}(\eta)$ contains $[2k, 2k+2]$ for some integer $k$, and by restricting the path if necessary, we may assume $\pi_{\gamma}(\eta) = \gamma([2k, 2k+2])$.

If $r(\eta)$ ever takes two values among $\{4^{2^{i}}  : i > 0\} \cap [R, +\infty)$, say $4^{2^{m}}$ and $4^{2^{m'}}$ for some $m < m'$, then the total variation of $r(\eta(t))$ is at least 
\[
4^{2^{m+1}} - 4^{2^{m}} = 4^{2^{m}} (4^{2^{m}} - 1) \ge R^{2}/2.\]
Consequently, we have $l(\eta) \ge 0.5R^{2}$.

If not, $r(\eta)$ takes at most one value $4^{2^{i}}$ among $\{4^{2^{j}} : j > 0\}$. If $i$ is even, then 
\[
F(\eta(t)) = F_{0}(\eta(t))
\]
for $t$ such that $\tau(\eta(t)) \in [2k+1.1, 2k+1.9]$. Since 
\[
F_{0}(\eta(t)) \ge e^{\kappa r(\eta(t))} \ge e^{ \kappa R},
\]
we have
\[
l(\eta) \ge \int F(\eta) \,d\tau(\eta) \ge e^{\kappa R} \times 0.8 = 0.8^{\kappa R}.
\]
Similarly, we have $l(\eta) \ge 0.8 e^{\kappa R}$ when $i$ is odd. This concludes that $\gamma$ is superlinear-divergent.
\end{proof}

Finally, we remark that superlinear divergence is invariant under quasi-isometry but the notion of strongly contracting is not. For example, let $X$ be the Cayley graph of a group $G$ equipped with the word metric associated to some finite generating set $\mathcal{S}$ and let $Z$ be a superlinear-divergent set in $X$. Then changing the generating set changes the metric in X by a quasi-isometry, and hence, $Z$ is still a superlinear-divergent set. But if $\gamma$ is a strongly contracting geodesic in $X$ it may not be strongly contracting with respect to the new metric. 

As an explicit example, it was shown in \cite[Theorem C]{SZ} that each mapping class group admits a proper cobounded action on a metric space $X$ such that all pseudo-Anosov elements have strongly contracting quasi-axes in X. To contrast, it was shown in \cite[Theorem 1.4]{RV} that the the mapping class group of the five-times punctured sphere can be equipped with a word metric such that the axis of a certain pseudo-Anosov map in the Cayley graph is not strongly contracting.

\medskip
\bibliographystyle{alpha}
\bibliography{superdivergent}

\end{document}